%% file: main.tex
\documentclass{article}
\usepackage{authblk}

\title{Weak convergence of quasi-birth-and-death processes with rational arrival process components to fluid queues}
\author[a]{Nigel Bean}
\author[a]{Angus Lewis\thanks{Corresponding author: angus.lewis@adelaide.edu.au}}
\affil[a]{The University of Adelaide}
\date{July 28, 2026}

\usepackage{amssymb,amsmath,amsthm}
\usepackage[mathscr]{euscript}
\usepackage{graphicx}
\usepackage{multirow}
\usepackage{tikz}
\usepackage{enumitem}

\usepackage[a4paper, total={6in, 8in}]{geometry}

\newtheorem{theorem}{Theorem}[section]
\newtheorem{corollary}[theorem]{Corollary}
\newtheorem{lemma}[theorem]{Lemma}

\newtheorem{remark}[theorem]{Remark}

\newcounter{subassumption}[theorem]

\makeatletter
\renewcommand{\p@subassumption}{\theasu}
\makeatother

\newcounter{subproperty}[theorem]

\makeatletter
\renewcommand{\p@subproperty}{\theproperty}
\makeatother

\newcommand{\wrt}{\, \mathrm{d}} 
 
\newcommand{\var}{\mathrm{Var}}
\newcommand{\calD}{\mathcal{D}}

\newcommand{\bs}{\boldsymbol}

\newcommand{\vligne}[1]{\begin{bmatrix} #1 \end{bmatrix}}

\newcommand{\tr}{^{\mbox{\tiny T}}}
\newcommand{\rapphase}{\phi}

\newcommand{\sign}{\mathrm{sgn}}

\begin{document}
\maketitle{}

\input{abstract.tex}
\input{intro.tex}

\input{preliminaries.tex}

\input{QBDRAPApproxmiation.tex}
\input{closingoperator.tex}
\input{convergence.tex}

\input{conclusion.tex}

\section*{Acknowledgements}
The authors thank Peter Taylor at The University of Melbourne for his contributions to the initial development of this work.

\section*{Funding}
This work was supported by the Australian Government Research Training Program (RTP) Scholarship and by the ARC Centre of Excellence in Mathematical and Statistical Frontiers (ACEMS, CE140100049).

\section*{Declaration of generative AI use}
The original version of this article was drafted by the authors. Claude (Sonnet 5) was used for copyediting and language refinement, to make the article more concise and readable. The corresponding author confirms, to the best of their knowledge, that the content in this article reflects the authors' own work and is accurate, and the corresponding author takes full responsibility for the integrity of the whole content.

\bibliographystyle{abbrv}
\bibliography{bibdata}

\appendix
\input{proof.tex}

\end{document}

%% file: abstract.tex
\begin{abstract}
    \label{sec:abstract}
    In this paper we construct a new approximation to a fluid queue as a quasi-birth-and-death process with rational arrival process components (QBD-RAP) and prove its convergence. Fluid queues are stochastic processes that move linearly at a rate governed by the state of a continuous-time Markov chain (CTMC), and are widely used to model telecommunications, power, risk, and storage systems. A key motivating application is to \emph{fluid-fluid queues}, whose analysis proceeds via operator-analytic expressions involving the generator of the underlying fluid queue; these expressions are differential operators that are not, in general, readily computable, so approximation is needed. Existing approximations with a probabilistic interpretation guarantee valid probabilities but require a fine discretisation to be accurate, while methods such as the Discontinuous Galerkin approach are more accurate at a given discretisation level but can produce negative mass or probabilities exceeding one. Our (QBD-RAP) approximation addresses this. Because the QBD-RAP is itself a stochastic process, the approximation it produces automatically retains the defining properties of a probability, while promising improved numerical accuracy over existing probabilistic schemes for a given discretisation level. We prove that the generator of the QBD-RAP converges to the generator of the fluid queue, which is, to our knowledge, the first generator-theoretic convergence result for a process with rational arrival process components. The proof introduces a new technique for the analysis of RAP-modulated processes, analysing the generator via bases of conditional residual time distributions rather than the orbit process used in prior RAP analyses, and along the way establishes that the phase process of the QBD-RAP and of the fluid queue share the same distribution.

    Keywords: fluid queue, quasi-birth-and-death process, rational arrival process, weak convergence.
\end{abstract}

%% file: intro.tex
\section{Introduction}
\label{sec:Introduction}
A fluid queue is a stochastic process, \(\{\bs X(t)\}_{t\geq0} = \{(X(t), \varphi(t))\}_{t\geq0}\), where \(\{\varphi(t)\}_{t\geq0}\) is a continuous-time Markov chain (CTMC) with finite state space \(\mathcal S\), and \(\{X(t)\}_{t\geq0}\) is the level process. Each state \(i\in\mathcal S\) is associated with a rate \(c_i\in\mathbb R\). When \(\varphi(t)=i\) the level process moves linearly (assuming \(X(t)\) is not at a boundary) with rate \(\cfrac{\wrt}{\wrt t}X(t)=c_i\). When modelling systems with fluid queues, the phase process typically represents some underlying operating state or environment of the system and \(X(t)\) is typically the main variable of interest.

Fluid queues are attractive modelling tools precisely because they are flexible without sacrificing tractability: \(\mathcal S\) can have any finite number of states, the rates \(c_i,\,i\in\mathcal S\), can take any real value, and extensions incorporating boundary conditions (for example, \cite{bean2009}) or level-dependent behaviour (for example, \cite{bo2008}) remain mathematically tractable. This combination has made fluid queues a useful modelling tool across telecommunications \cite{anick1982}, power systems \cite{hydro}, risk processes \cite{betal2005}, environmental modelling \cite{wurm2020}, and dam storage \cite{loynes1962}.

Despite this tractability, a key motivating application, the analysis of \emph{fluid-fluid queues} \cite{bo2014}, proceeds via operator-analytic expressions written in terms of the fluid queue's generator, and these expressions are differential operators that are not, in general, readily computable. Approximation is therefore needed. Applying our construction to fluid-fluid queues yields a fluid process with rational arrival process components \cite{bgnp2021} that retains a stochastic interpretation, and therefore automatically inherits correct probabilistic properties (e.g., non-decreasing distribution functions, conservation of mass). These are properties that are not guaranteed by every approximation method.

Two previous approximations have been developed that are suitable for the application to fluid-fluid queues. The QBD approximation of \cite{bo2013} has a probabilistic interpretation and so guarantees valid probabilities (non-negative, bounded above by 1, and integrating to 1) but requires a fine discretisation to be accurate. The Discontinuous Galerkin (DG) approximation of \cite{blnos2022} is far more accurate for smooth problems at a given discretisation level, but can oscillate at discontinuities badly enough to produce negative mass or probabilities exceeding 1.

Fluid queues have been well studied, with expressions available for stationary distributions and probabilities (or Laplace transforms of probabilities) of certain transient events, such as hitting times and first-return times \cite{ajr2005,ar2004,bean2005,latouche2018,ramaswami1999matrix,dasilva2005}. As argued in \cite{blnos2022}, most of the existing analyses do not give expressions which are suitable for the application to fluid-fluid queues as they do not capture necessary spatial information. For example, \cite{ramaswami1999matrix} constructs a QBD to evaluate the matrix of first-return probabilities of a fluid queue, but there is no direct mapping between the level process of the QBD and the level process of the fluid queue. Similarly, the QBD in \cite{ar2004} represents the workload in the queue and not the level itself. 

In this paper we develop a new approximation that retains the probabilistic guarantees of \cite{bo2013} while promising improved numerical accuracy for a given discretisation level. This paper establishes the technical foundations: the construction of the QBD-RAP and a full proof of its convergence to the fluid queue, with numerical results demonstrating its accuracy to follow in subsequent work (see also \cite{a2023}).

There are three main contributions of this paper: (1) the construction of the QBD-RAP process itself, (2) the construction of the closing operator, which uses bases of conditional residual time distributions to recover an approximation to the distribution of the fluid queue from the level and orbit processes of the QBD-RAP, and (3) distributional results for the QBD-RAP. The main distributional result is a weak convergence of the approximation scheme and is constructed via convergence of generators, utilising the aforementioned closing operators. We also show that the phase processes of the QBD-RAP and fluid queue have the same distribution.

To the authors' knowledge, this is the first generator-theoretic convergence result for a process with RAP components. Typical analysis of processes with RAP components is carried out using the \textit{orbit process} \cite{ab1999,bgnp2021,bo2013}, and in particular, by computing the expectation of the orbit process on certain events. The novel approach used here, for both the closing operator and the convergence proof, does not use the orbit process directly and instead constructs bases of conditional residual time distributions to analyse the process and its generator. One advantage of using the conditional residual time distributions is that they have a direct probabilistic interpretation, whereas the probabilistic interpretation of the orbit process is more abstract. 

The structure of the paper is as follows. Section \ref{sec:Preliminaries} introduces fluid queues and provides background on matrix exponential distributions and QBD-RAPs. Section \ref{sec:QBD-RAP Approximation} constructs the QBD-RAP approximation to the fluid queue. Section \ref{sec:closing operator} introduces the closing operator, a tool used to approximate the distribution of the fluid queue from the orbit process of the QBD-RAP. Section \ref{sec:Convergence of the QBD-RAP approximation} proves the convergence of the QBD-RAP to the fluid queue. Section \ref{sec:Conclusion} makes concluding remarks. Some technical results are left to the appendix.

%% file: preliminaries.tex
\section{Preliminaries}
\label{sec:Preliminaries}

\subsection{Fluid Queues}
Denote a fluid queue by \(\{\bs X(t)\}_{t\geq0} = \{(X(t), \varphi(t))\}_{t\geq0}\), where \(\{\varphi(t)\}_{t\geq0}\) is called the phase process and is a continuous-time Markov chain (CTMC) with finite state space \(\mathcal S\), and \(\{X(t)\}_{t\geq0}\) is called the level process. Denote the generator of the phase process by \(\bs T=[T_{ij}]_{i,j\in\mathcal S}\). Each state in \(\mathcal S\) has an associated rate \(c_i\in\mathbb R,\, i \in\mathcal S\). In this paper, for simplicity, we assume \(c_i\neq 0\); however, the results here can be extended to the case \(c_i=0\) (see \cite{a2023}, although a different resolvent-based proof technique is used there). On the event \(\{\varphi(t)=i\}\) the level process moves linearly (assuming \(X(t)\) is not at a boundary) with rate \(\cfrac{\wrt}{\wrt t}X(t)=c_i\). 

The level process can have a bounded, semi-bounded, or unbounded domain. In this paper we consider bounded fluid queues only as these are practical when implementing the approximation numerically. Without loss of generality, assume \(\{X(t)\}_{t\geq0}\) has a lower boundary at \(0\) and an upper boundary at \(b\), \(0<b<\infty\). For simplicity, we assume \textit{regulated} boundaries with the following behaviour (see \cite{a2023} for a treatment of other boundary behaviours). Let \(\mathcal N\) (respectively, \(\mathcal P\)) be the set of states \(i\in\mathcal S\) such that \(c_i<0\) (respectively, \(c_i>0\)). When the level process \(\{X(t)\}_{t\geq0}\) hits the lower boundary \(0\) (upper boundary \(b\)) and does so when the phase is \(i\in\mathcal N\) (\(i\in\mathcal P\)), then the level process remains at the boundary while the phase remains in \(\mathcal N\) (\(\mathcal P\)) until the first time that the phase process enters \(\mathcal P\) (\(\mathcal N\)). The evolution of the level process can be expressed mathematically as follows.
\[\cfrac{\wrt}{\wrt t} X(t) = \begin{cases} c_{\varphi(t)}, & \mbox{ if } 0<X(t)<b, \\ \max\{0,c_{\varphi(t)}\}, & \mbox{ if } X(t)=0, \\ \min\{0,c_{\varphi(t)}\}, & \mbox{ if } X(t)=b.  \end{cases}\]

Let \(\mathcal L\) denote the vector space of bounded, real-valued functions on \([0,b]\times \mathcal S\), equipped with the supremum norm \(||\cdot||_{\infty}\). Thus, \(\mathcal L\) is a complete normed vector space and therefore a Banach space. For \(f\in\mathcal L\), \(f:[0,b]\times \mathcal S\to \mathbb R\) the transition semigroup of the fluid queue is the family of operators \(\{T(t)\}_{t\geq 0}\) defined by 
\[T(t)f(x,i)=\mathbb E[f(X(t),\varphi(t))\mid X(0)=x,\varphi(t)=i].\]
The generator of a fluid queue is the differential operator, \(B\), with
\[Bf(x,i) = \sum_{j\in\mathcal S}T_{ij}f(x,j) + c_i\cfrac{\wrt}{\wrt x}f(x,i),\]
along with boundary conditions \(Bf(0,i)=\sum_{j\in\mathcal S}T_{ij}f(0,j),\) \(i\in\mathcal N\) and \(Bf(b,i)=\sum_{j\in\mathcal S}T_{ij}f(b,j),\) \(i\in\mathcal P\). The domain of the generator, \(\mathrm{Dom}(B)\), is the set of functions for which \(B\) above is well-defined, i.e.,~the subset of functions \(f\in\mathcal L\) for which \(f\) is continuously differentiable with respect to the first (spatial, \(x\)) variable on \((0,b)\) for each \(i\in\mathcal S\), with \(f_x(\cdot,i)\) extending continuously on \([0,b]\). We show convergence of the QBD-RAP approximation via convergence of the generator of the QBD-RAP to the generator, \(B\), of the fluid queue.

For future reference, define the diagonal matrices of rates \(\bs C=\mbox{diag}(|c_i|, i\in\mathcal S)\), \(\bs C_-=\mbox{diag}(|c_i|, i\in\mathcal N)\), \(\bs C_+=\mbox{diag}(|c_i|, i\in\mathcal P)\), and the sub-generator matrices \(\bs T_{++}=[T_{ij}]_{i,j\in\mathcal P}\), \(\bs T_{+-}=[T_{ij}]_{i\in\mathcal P,j\in\mathcal N}\), \(\bs T_{-+}=[T_{ij}]_{i\in\mathcal N, j\in\mathcal P}\), and \(\bs T_{--}=[T_{ij}]_{i,j\in\mathcal N}\).

\subsection{Matrix exponential distributions}
The approximation scheme described in this paper is based on matrix exponential distributions. A random variable, \(Z\), has a matrix exponential distribution if it has a distribution function of the form \(1-\bs \alpha e^{\bs Sx}(-\bs S)^{-1}\bs s\), where \(\bs \alpha\) is a \(1\times p\) \emph{initial vector}, \(\bs S\) a \(p\times p\) matrix, and \(\bs s\) a \(p\times 1\) \emph{closing vector}, and \(\displaystyle e^{\bs S x} = \sum_{n=0}^\infty \cfrac{\left(\bs Sx\right)^n}{n!}\) is the matrix exponential. The density function of \(Z\) is given by \(f_Z(x) = \bs\alpha e^{\bs S x}\bs s\). The parameters \((\bs \alpha, \bs S, \bs s)\) must satisfy \(\bs \alpha e^{\bs S x} \bs s\geq 0,\) for all \(x\geq 0\) and \(\lim_{x\to\infty} 1-\bs \alpha e^{\bs Sx}(-\bs S)^{-1}\bs s = 1\) so that \(f_Z\) is a density function. In general there is the possibility of an \emph{atom} (a point mass) at 0, but here we do not consider this possibility. The class of matrix exponential distributions is an extension of Phase-type distributions where, for the latter, \(\bs S\) must be a sub-stochastic generator matrix of a CTMC, \(\bs s = -\bs S\bs e\) where \(\bs e\) is a \(p\times 1\) vector of ones, and \(\bs \alpha\) is a discrete probability distribution.  

A \emph{representation} of a matrix exponential distribution is a triplet \((\bs \alpha, \bs S, \bs s)\), and we write \(Z\sim ME(\bs \alpha, \bs S, \bs s)\) to denote that \(Z\) has a matrix exponential distribution with this representation. Representations of matrix exponential distributions are not unique \cite[Section~3.1]{MEinAP}. The \emph{order} of the representation is the dimension of the square matrix \(\bs S\), i.e.~if \(\bs S\) is \(p\times p\) the representation is said to be of order \(p\). A representation is called \emph{minimal} when \(\bs S\) has the smallest possible dimension and a minimal representation can always be found for a matrix exponential distribution \cite[Corollary 4.2.8]{MEinAP}. Throughout this paper we assume that the representation of any matrix exponential distribution is minimal. 

Let \(\bs e_i\) be a vector with a 1 in the \(i\)th position and zeros elsewhere. Throughout this paper we assume that the closing vector of a matrix exponential distribution is given by \(\bs s = -\bs S\bs e\), and that \((\bs e_i\tr{},\bs S,\bs s)\) for \(i=1,\dots,p\) are representations of matrix exponential distributions. It is always possible to find such a representation \cite[Theorem 4.5.17, Corollary 4.5.18]{MEinAP}. As such, we abbreviate our notation \(Z\sim ME(\bs \alpha, \bs S, \bs s)\) to \(Z\sim ME(\bs \alpha, \bs S)\).

For numerical implementation of the QBD-RAP approximation the intention is to use \emph{concentrated matrix exponential distributions} (CMEs) \cite{hstz2016,hht2020} in the construction. The class of concentrated matrix exponential distributions is a certain set of matrix exponential distributions with a very low coefficient of variation (variance relative to the mean). The class of concentrated matrix exponential distributions (CMEs) is found numerically in \cite{hht2020}. \cite{hht2020} demonstrate numerically that the variance of concentrated matrix exponential distributions decreases at rate approximately \(\mathcal O(1/p^2)\) as \(p\) increases, where \(p\) is the order of the representation. For comparison, the variance of an Erlang distribution, the most concentrated Phase-type distribution, decreases at rate \(\mathcal O(1/p)\) as the order \(p\) increases \cite{as1987}. Similar to the idea of Erlangization (see, for example, \cite{aau2002, he2023}), in this paper we use CMEs to approximate the distribution of deterministic events. Since CMEs have a lower coefficient of variation than Erlang distributions then, intuitively, we expect that the QBD-RAP scheme will have smaller approximation errors than the QBD approximation \cite{bo2013}. Note, that the arguments for convergence are not dependent on properties specific to CMEs.

\subsection{QBD-RAPs}\label{sec: qbd-rap}
QBD-RAPs are formally introduced in \cite{bn2010} as extensions of QBDs which allow matrix exponential times between level changes. QBD-RAPs are built from a marked (Batch) rational arrival process (BRAP), itself an extension of the rational arrival process (RAP) of \cite{ab1999} (see also Section~10.5 of \cite{MEinAP}). Informally, a RAP is a point process \(N\) for which there is a finite-dimensional space of measures \(V\) such that, for any \(t\), the conditional distribution of the process after time \(t\) given its history up to \(t\) always lies in \(V\). A BRAP extends this to marked point processes, with marks taking values in a countable set \(\mathscr K\subset\mathbb Z\), where each mark type is required to occur infinitely often. We do not use the formal (measure-theoretic) definitions of RAPs and BRAPs directly; rather, we use the following characterisation of a BRAP, analogous to the characterisation of a RAP given in \cite{ab1999}.

Let \(N\) be a marked point process on \((0,\infty)\) with event times \(Y_0=0<Y_1<Y_2\cdots\), counting process \(\{N(t)\}_{t\geq0}\), and shift operator \(\theta_t N=\{N(t+s)\}_{s\geq0}\). Let \(\mathcal F_t\) be the \(\sigma\)-algebra generated by \(N(u),\,u\in[0,t]\). Suppose the mark of the \(n\)th event is \(M_n\), taking values in \(\mathscr K\), and for \(i\in\mathscr K\) let \(\{N_i(t)\}_{t\geq0}\) be the counting process of events with mark \(i\). Denote by \(f_{N,n}(y_1,m_1,y_2,m_2,\dots,y_n,m_n)\) the joint density/probability mass function of the first \(n\) inter-arrival times, \(Y_1,Y_2-Y_1,\dots,Y_n-Y_{n-1}\), and the associated marks \(M_1,...,M_n\). For a real square matrix \(\bs E\) of size \(p\times p\) with eigenvalues \(\lambda_i,\,i=1,...,p\), such that \(Re(\lambda_1)\geq Re(\lambda_2)\geq ... \geq Re(\lambda_p)\), we define \(dev(\bs E)=\lambda_1\), the \textit{dominant real eigenvalue} of \(\bs E\). For generator matrices of ME random variables that are minimal, e.g., the matrix \(\bs S\), we can always choose the dominating eigenvalue to be real \cite[Theorem~4.1.4]{MEinAP}.
\begin{theorem}[Theorem~1 of \cite{bn2010}]\label{thm:qbdrapchar}
A process \(N\) is a BRAP if there exist matrices \(\bs S\), \(\bs D_i,\,i\in\mathscr K\), and a row vector \(\bs \alpha\) such that \(dev(\bs S)<0\), \(dev(\bs S+\bs D)=0\), \((\bs S+\bs D)\bs e = 0\), \(\bs D = \sum_{i\in\mathscr K} \bs D_i\), and 
\begin{align}f_{N,n}(y_1,m_1,y_2,m_2,\dots,y_n,m_n) = \bs \alpha e^{\bs S y_1}\bs D_{m_1} e^{\bs S y_2} \bs D_{m_2}\dots e^{\bs S y_n} \bs D_{m_n}\bs e.\label{eqn: brap}\end{align}
Conversely, if a point process has the property (\ref{eqn: brap}) then it is a BRAP.
\end{theorem}

Denote such a process \(N\sim BRAP(\bs \alpha, \bs S, \bs D_i,\,i\in\mathscr K).\) 

In the construction of the QBD-RAP approximation scheme we utilise the \textit{orbit process} description of the BRAP, which we now recall for completeness. From \cite{bn2010}, associated with a BRAP is a row-vector-valued {orbit} process, \(\{\bs A(t)\}_{t\geq0},\)
\[\bs A(t) = \cfrac{\bs \alpha\left( \prod\limits_{i=1}^{N(t)} e^{\bs S(Y_{i}-Y_{i-1})}\bs D_{M_i}\right)e^{\bs S (t-Y_{N(t)})}}{\bs \alpha\left( \prod\limits_{i=1}^{N(t)} e^{\bs S(Y_{i}-Y_{i-1})}\bs D_{M_i}\right)e^{\bs S( t-Y_{N(t)})}\bs e}.\]
Thus, \(\{\bs A(t)\}\) is a piecewise-deterministic Markov process which, in between events, evolves deterministically according to 
\[\bs A(t) = \cfrac{\bs A(Y_{N(t)})e^{\bs S(t-Y_{N(t)})}}{\bs A(Y_{N(t)})e^{\bs S(t-Y_{N(t)})}\bs e}.\]
Equivalently, \(\{\bs A(t)\}\) evolves deterministically according to the differential equation
\[\cfrac{\wrt}{\wrt t} \bs A(t) = \bs A(t)\bs S - \bs A(t)\bs S\bs e\cdot \bs A(t),\] 
between events. 

The process \(\{\bs A(t)\}\) can jump at event times of \(N\) (the process may, but does not necessarily, jump at these times, but we refer to it as a `jump' and typically the dynamics change at this point). At time \(t\) the intensity with which \(\{\bs A(t)\}\) has a jump is \(\bs A(t) \bs D\bs e\), i.e.~\(\mathbb P(N(t)=n,N(t+dt)=n+1)=\bs A(t) \bs D\bs e\wrt t\). Upon an event at some time \(t\), the event is associated with mark \(i\) with probability \(\bs A(t) \bs D_i\bs e/\bs A(t) \bs D\bs e\). Upon an event at time \(t\) with mark \(i\), the new position of the orbit is \(\bs A(t) = \bs A(t^-) \bs D_i/\bs A(t^-) \bs D_i\bs e\). Thus, jumps of the orbit process are {linear} transformations of the orbit process at the time immediately before the jump. 

A key property of BRAPs is that the distribution of the process, conditional on the orbit process up to and including time \(t\), can be expressed in terms of \(\bs A(t)\) and measures \(v_1,...,v_p\in V\) which form a basis for the space \(V\); 
\begin{align}\label{eqn: rapproperty}
	\mu(t,\cdot)=\mathbb P(\theta_tN\in \cdot\mid \mathcal F_t)=\bs A(t)\bs v(\cdot),
\end{align} 
where \(\bs v = [v_1,...,v_p]\tr\).

BRAPs are an extension of marked Markovian arrival processes to include matrix exponential inter-arrival times. For marked MAPs, the vector \(\bs A(t)\) is a vector of posterior probabilities of a continuous-time Markov chain. 

Intuitively, \(\bs A(t)\) encodes all the information about the event times of the BRAP and the associated marks up to time \(t\), that is needed to determine the future behaviour of the point process. Let \(\mathcal F_{t}\) be the \(\sigma\)-algebra generated by \(N(u), u\in[0,t]\), then the distribution of \(N\) given \(\mathcal F_t\) is equivalent to the distribution of \(N\) given \(\bs A(t)\), and the distribution is \(N|\bs A(t) \sim BRAP(\bs A(t), \bs S, \bs D_i,i\in\mathscr K)\).

Consider a BRAP, \(N\sim BRAP(\bs \alpha, \bs S, \bs D_i,\,i\in\{-1,0,+1\}),\) then the process \(\{(L(t),\bs A(t))\}_{t\geq0}\) formed by letting \(L(t) = N_{+1}(t) - N_{-1}(t)\) is a QBD-RAP \cite{bn2010}.

%% file: QBDRAPApproxmiation.tex
\section{QBD-RAP Approximation}
\label{sec:QBD-RAP Approximation}
In this section we describe the approximation and its dynamics. 

\subsection{Construction}\label{subsec: Construction}
Consider a regular grid of points \(y_k = (k-1)\Delta,\,\Delta = b/K,\, k=1,...,K+1\) in the interval \([0,b]\) (the restriction to a regular grid of points is for convenience, in general, the grid does not need to be regular, see \cite{a2023}). For states \(i\in\mathcal N\), partition the interval \([0,b]\) into the set \(\mathcal D_{0,-} = \{0\}\) and regular intervals \(\mathcal D_{k,-} = (y_k, y_{k+1}],\, k=1,...,K\). Similarly, for states \(i\in\mathcal P\) partition the interval \([0,b]\) into the set \(\mathcal D_{K+1,+} = \{b\}\) and regular intervals \(\mathcal D_{k,+} = [y_k, y_{k+1}),\, k=1,...,K\). For simplicity, introduce the notation \(\mathcal D_{k,i}=\mathcal D_{k,\sign{(c_i)}}\).

Denote the approximating QBD-RAP by \(\{(L(t), \bs A(t), \rapphase(t))\}_{t\geq0}\) where \(\{L(t)\}\) is the level process, \(\{\bs A(t)\}_{t\geq0}\) is the orbit process, and \(\{\rapphase(t)\}_{t\geq0}\) is the phase process. The level approximates the process 
\begin{align}
    \displaystyle \left\{\sum_{k\in{0,...,K+1}} k1(X(t)\in\calD_{k,\varphi(t)})\right\}_{t\geq 0}. \label{eqn: kfj}
\end{align}
We show later that the phase process \(\{\rapphase(t)\}_{t\geq0}\) is equal in distribution to the phase process of the fluid queue, \(\{\varphi(t)\}_{t\geq0}\). The value of the orbit process \(\bs A(t)\) can be interpreted to (approximately) capture information about the density of the position of \(X(t)\) within the intervals \(\mathcal D_{k,m},\,m\in\{+,-\},\,k\in\{0,...,K+1\}\). More on this later too.

We approximate the dynamics of the fluid queue in intervals using matrix exponential distributions with low coefficient of variation (e.g., CMEs) to approximate deterministic events. So, throughout, consider a matrix exponential random variable \(Z\sim ME(\bs \alpha, \bs S)\) with mean \(\Delta\) and small variance. Recall that we can choose the ME representation \((\bs\alpha, \bs S, \bs s)\) such that \(\bs\alpha = \bs e_1\tr{}\) and \(\bs e_n \tr{} e^{\bs Sx}\bs s,\, n=1,...,p\) are densities of matrix exponential distributions. Specifically, by \cite[Theorem 4.5.17, Corollary 4.5.18]{MEinAP}, we choose the representation such that \(\bs e_n \tr{} e^{\bs Sx}\bs s,\, n=1,...,p\) are the densities of the distributions of the conditioned residual times, \(Z-t_n \mid Z\geq t_n\), for some \(0=t_1<2\varepsilon<t_2<...<t_p<\Delta - \varepsilon\) for \(\varepsilon>0\) sufficiently small, i.e.
\[\mathbb P(Z-t_n\geq x\mid Z>t_n)=\bs e_n\tr{} e^{\bs Sx}\bs e,\,x\geq 0.\] 
The times \(0=t_1<2\varepsilon<t_2<...<t_p<\Delta - \varepsilon\) must be chosen such that \(x\mapsto \bs e_n\tr{} e^{\bs Sx}\bs e,\, n=1,...,p,\) are linearly independent. Suitable times \(0=t_1<2\varepsilon<t_2<...<t_p<\Delta - \varepsilon\) can always be found by choosing \(t_n,\, n=2,...,p,\) relatively prime to \(2\pi \sigma_k,\,k=1,...,q<p/2\) where \(\sigma_k\) are the magnitudes of the imaginary parts of the \(q\) conjugate pairs of complex eigenvalues of \(\bs S\) \cite[Remark 4.5.19]{MEinAP}. In the limit as \(p\to\infty\), there are at most countably many points that are not relatively prime to \(2\pi \sigma_k,\,k=1,...,q<p/2\), so there are uncountably many choices for \(t_1,...,t_p.\)

For the proof of convergence, we further restrict the choice of \(t_1,...,t_p\). Specifically, for a given \(\delta>0\) and assuming \(p\) is sufficiently large and \(\var(Z)\) is sufficiently small, we choose \(t_1,...,t_p\) such that \(t_1=0<2\var(Z)^{1/3} < t_2 < t_3 < ... < t_p < \Delta-\var(Z)^{1/3}<\Delta=t_{p+1}\), also \(\max\limits_{n=1,...,p}|t_n-t_{n+1}|<\frac{1}{2}\delta\), and \(\cfrac{\bs \alpha e^{\bs S t_n}\bs s}{\bs \alpha e^{\bs S t_n}\bs e} \leq \cfrac{\var(Z)^{1/3}}{{\frac{1}{4}\delta}\,(1-\var(Z)^{1/3})}\). The following result states that this is always possible, provided the variance \(Z\) is sufficiently small and the order of its minimal representation is sufficiently large.
\begin{lemma}\label{lem: points}
	Let \(Z\) be a matrix exponential random variable. Given \(\delta>0\), if \(\var(Z)\) is sufficiently small and the order of the minimal representation of \(Z\) is sufficiently large, then we can choose \(t_1,...,t_p\), such that \(t_1,...,t_p\), are relatively prime to \(2\pi\sigma_i\), \(i=1,...,q\), where \(\sigma_i\) is the magnitude of the imaginary part of the complex conjugate pairs of eigenvalues of \(\bs S\), and
	\begin{align}
		2\var(Z)^{1/3} < t_2 < t_3 < ... < t_p < \Delta-\var(Z)^{1/3}<\Delta=t_{p+1},\label{eqn: cond 1}
	\end{align}
	\begin{align}
		\max\limits_{n=1,...,p}|t_n-t_{n+1}|<\frac{1}{2}\delta,\label{eqn: cond 2}
	\end{align}
	and 
	\begin{align}
		\cfrac{\bs \alpha e^{\bs S t_n}\bs s}{\bs \alpha e^{\bs S t_n}\bs e} < \cfrac{\var(Z)^{1/3}}{{\frac{1}{4}\delta}\,(1-\var(Z)^{1/3})}\label{eqn: cond 3}
	\end{align}
	are satisfied.
\end{lemma}
\begin{proof}
	Let \(I\subset [2\var(Z)^{1/3}, \Delta-\var(Z)^{1/3}]\) be any closed interval of length \(\frac{1}{4}\delta\). By the Mean Value Theorem, there exists \(r\) in the interior of \(I\) such that 
	\[
		\cfrac{\int_{I} \bs \alpha e^{\bs S t}\bs s\wrt t}{\frac{1}{4}\delta} = \bs \alpha e^{\bs S r}\bs s,
	\]
	 and by Chebyshev's inequality \(\int_{I} \bs \alpha e^{\bs S t}\bs s\wrt t< \var(Z)^{1/3}\), where the inequality is strict since \(Z\) is not degenerate or supported on only two points, and therefore \(\bs \alpha e^{\bs S r}\bs s< \cfrac{\var(Z)^{1/3}}{{\frac{1}{4}\delta}}\). Since \(t\mapsto \bs \alpha e^{\bs S t}\bs s\) is a continuous function there exists an uncountable number of points in \(I\) such that \(\bs \alpha e^{\bs S t}\bs s< \cfrac{\var(Z)^{1/3}}{{\frac{1}{4}\delta}}\). Moreover, since there are at most countably many points that are not relatively prime to \(2\pi\sigma_i\), \(i=1,...,q\), then there are uncountably many points in \(I\) that are relatively prime to \(2\pi\sigma_i\), \(i=1,...,q\). Hence, in any interval of length \(\frac{1}{2}\delta\) (two adjacent intervals of length \({\frac{1}{4}\delta}\)), there is an uncountable number of pairs \(t_n, t_{n+1}\) with \(\frac{1}{4}\delta < |t_n-t_{n+1}|< \frac{1}{2}\delta\) such that \(\bs \alpha e^{\bs S t_n}\bs s< \cfrac{\var(Z)^{1/3}}{{\frac{1}{4}\delta}}\) and \(\bs \alpha e^{\bs S t_{n+1}}\bs s < \cfrac{\var(Z)^{1/3}}{{\frac{1}{4}\delta}}\) that are also relatively prime to \(2\pi\sigma_i,\,i=1,...,q\). This establishes the constraints of the lemma are satisfied for a single pair of points. To establish the constraints hold for the whole sequence of points \(t_1,...,t_p\), then partition \([2\var(Z)^{1/3}, \Delta-\var(Z)^{1/3}]\) into disjoint intervals of length \(\delta/4\), and then pick \(t_2,...,t_p\) from each interval such that the constraints are satisfied pairwise noting also that \(t_2<\delta/2\) and \(\Delta-t_p<\delta/2\) hold since \(\var(Z)^{1/3}\) is small relative to \(\delta\), so that \(t_1=0\) and \(t_{p+1}=\Delta\) satisfy \eqref{eqn: cond 2} as well. This completes the proof.
\end{proof}

In the rest of this section we describe the construction and dynamics of the QBD-RAP process.

\subsubsection{Up to the first event after entering a level}\label{subsubsec: first event}
Suppose first that \(X(t)=y_k\) and \(\varphi(t)=i\in\mathcal P\) and let \(\mathcal E(t)=\{\varphi(s)=i,\, \forall \, s\geq t\}\) (respectively, \(\mathcal G(t)=\{\rapphase(s)=i,\, s\geq t\}\)) be the event that the phase of the fluid queue (respectively, QBD-RAP) is constant. On \(\mathcal E(t)\), the process \(\{X(t)\}_{t\geq0}\) leaves \(\mathcal D_{k,+}\) at exactly time \(t + \Delta/|c_i|\). To approximate this deterministic time, consider the random variable \(Z/|c_i|\). The random variable \(Z/|c_i|\) is concentrated at the time \(\Delta/|c_i|\) since
\begin{align*}
	\mathbb P(Z/|c_i| \in{}((\Delta-\varepsilon)/|c_i|, (\Delta+\varepsilon)/|c_i|) ) 
	&= \mathbb P(Z\in{}(\Delta-\varepsilon, \Delta + \varepsilon))
	\geq 1-\cfrac{\var(Z)}{\varepsilon^2},
\end{align*}
by Chebyshev's inequality and by choosing \(\varepsilon = \var(Z)^{1/3}\) (similar to the choice of \(\varepsilon\) in the proof of Theorem~4 of \cite{hhat2020}), then 
\[\mathbb P(Z_i \in{}((\Delta-\varepsilon)/|c_i|-u, (\Delta+\varepsilon)/|c_i|-u) )\geq 1-\var(Z)^{1/3} \approx 1, \]
when \(\var(Z)\) is small. Hence, for our QBD-RAP approximation, we suppose that the distribution of time that the level process, \(\{L(t)\}_{t\geq0}\), remains at level \(k\) on the event \(\mathcal G(t)\), given \(\{L(t)\}_{t\geq0}\) just entered level \(k\) in Phase~\(i\in\mathcal P\), has the distribution of \(Z/|c_i|\). 

Suppose again that \(X(t)=y_k\) and \(\varphi(t)=i\in\mathcal P\). After \(u\in [0,\Delta/|c_i|)\) amount of time has elapsed and on the event \(\mathcal E(t)\), then \(X(t+u) = y_k + c_i u \in\calD_{k,+}\) and, moreover, \(\{X(t)\}_{t\geq0}\) will leave \(\calD_{k,+}\) at time \(t+\Delta/|c_i|-u\). Consider the residual time \(R_i(u) = (Z/|c_i|-u)1\{Z/|c_i|-u>0\}\) which is \(0\) if \(Z/|c_i|\leq u\), or the time until \(Z/|c_i|\) on the event that \(Z/|c_i|\) is greater than \(u\). The density function of \(R_i(u)\), given \(Z/|c_i|>u\geq 0\), is 
\[
    f_{R_i(u)}(r) = \cfrac{\bs\alpha e^{\bs S|c_i|(u + r) } \bs s|c_i|}{\bs\alpha e^{\bs S|c_i|(u) } \bs e},\quad r\geq 0.
\]
On the event \(Z/|c_i|>u\), the residual time \(R_i(u)\) is approximately \(\Delta/|c_i|-u\), since, for \(\varepsilon>0\) and \(u<(\Delta-\varepsilon)/|c_i|\), 
\begin{align*}
	\mathbb P(R_i(u) \in{}((\Delta-\varepsilon)/|c_i|-u, (\Delta+\varepsilon)/|c_i|-u) ) 
	&= \mathbb P(Z/|c_i|\in{}((\Delta-\varepsilon)/|c_i|, (\Delta+\varepsilon)/|c_i|) ) 
	\\&= \mathbb P(Z\in{}(\Delta-\varepsilon, \Delta + \varepsilon))
	\\&\geq 1-\cfrac{\var(Z)}{\varepsilon^2},
\end{align*}
which, upon choosing \(\varepsilon=\var(Z)^{1/3}\), gives  
\[
    \mathbb P(R_i(u) \in{}((\Delta-\varepsilon)/|c_i|-u, (\Delta+\varepsilon)/|c_i|) )\geq 1-\var(Z)^{1/3} \approx 1,
\]
when \(\var(Z)\) is small. Hence, for the QBD-RAP approximation, we suppose that the distribution of time that the level process, \(\{L(t)\}_{t\geq0}\), remains at level \(k\) on the event \(\mathcal G(t)\), given \(\{L(t)\}_{t\geq0}\) entered level \(k\) exactly \(u\) units of time earlier in Phase~\(i\in\mathcal P\), has the distribution \(R_i(u)\) on the event that \(Z/|c_i|>u\).

On the event \(\mathcal G(t)\), given \(\{L(t)\}_{t\geq0}\) entered level \(k\) exactly \(u\) units of time earlier in Phase~\(i\in\mathcal P\), the orbit position at time \(t+u\) is \(\cfrac{\bs \alpha e^{\bs S |c_i|u}}{\bs \alpha e^{\bs S |c_i|u}\bs e},\) and so we say that this orbit position \textit{corresponds} to position \(y_k+|c_i|u\), which is the level of the fluid queue at time \(t+u\) on the event \(\mathcal G(t)\), given \(\{X(t)\}_{t\geq 0}\) entered \(\mathcal D_{k,\varphi(t)}\) exactly \(u\) units of time earlier in Phase~\(i\in\mathcal P\).

Analogously for states with negative rate, \(\mathcal N\), we suppose that the distribution of time that the level process, \(\{L(t)\}_{t\geq0}\), remains at level \(k\) on the event \(\mathcal G(t)\), given \(\{L(t)\}_{t\geq0}\) entered level \(k\) exactly \(u\) units of time earlier in Phase~\(i\in\mathcal N\) has the distribution \(R_i(u)\) on the event that \(Z/|c_i|>u\); moreover, the orbit position at time \(t+u\) is \(\cfrac{\bs \alpha e^{\bs S |c_i|u}}{\bs \alpha e^{\bs S |c_i|u}\bs e},\) and we say that this orbit position corresponds to position \(y_{k+1}+c_iu\) (where the upper boundary of the interval, \(y_{k+1}\), is present because now the fluid queue enters the interval \(\mathcal D_{k,\varphi(t)}\) from the top).


To demonstrate that \(R_i(u)\) is a reasonable approximation to deterministic events, Figure~\ref{fig:residual distributions} gives an example of a density function of a concentrated matrix exponential random variable \(Z/|c_i|\) with mean \(\Delta=1\) and \(c_i=1\), as well as the density function of \(R_i(0.3)\) conditional on \(Z_i>0.3\) and \(R_i(0.6)\) conditional on \(Z_i>0.6\), for comparison. 

\begin{figure}
\centering
\includegraphics[width=\textwidth]{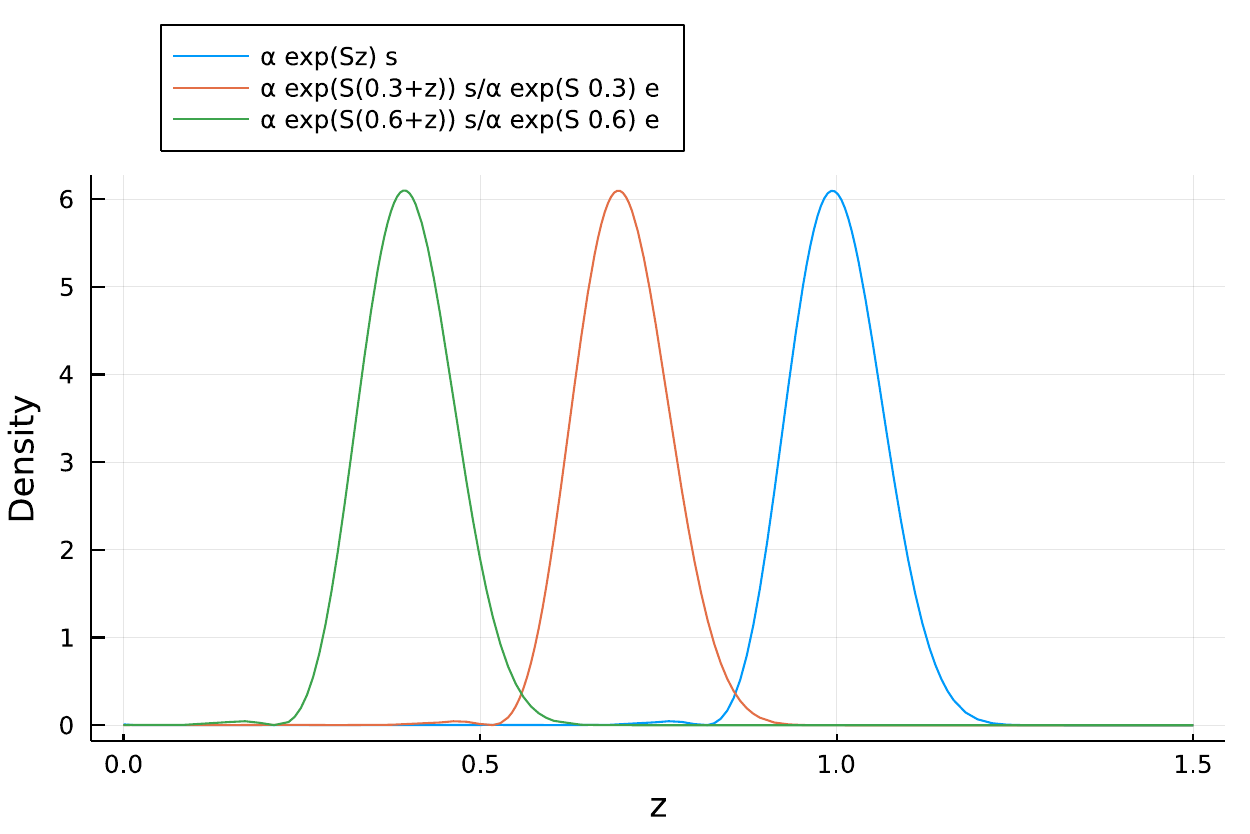}
\caption{The density function for a concentrated matrix exponential of order 21 from \cite{hht2020} (blue) and corresponding conditional density functions of the residual lives, \(R_i(0.3)\) given \(Z>0.3\) (red), and \(R_i(0.6)\) given \(Z>0.6\) (green). Observe how the density function of the \(Z/|c_i|\) (blue) approximates a point mass at \(\Delta=1\), while the density functions of \(R_i({0.3})\) given \(Z>0.3\) (red) and \(R_i({0.6})\) given \(Z>0.6\) (green) approximate point masses at \(0.7\) and \(0.4\), respectively. }
\label{fig:residual distributions}
\end{figure}

\subsubsection{On a change of phase}\label{subsubsec: change of phase}
We claimed above that the orbit position 
\begin{align}
	\bs A(u) = \cfrac{\bs \alpha e^{\bs S|c_i| u}}{\bs \alpha e^{\bs S|c_i| u}\bs e},\label{eqn: orbit i}
\end{align}
corresponds to the position of the fluid queue \(y_{k+1(i\in\mathcal N)}+c_iu\), where \(1(\cdot)\) is the indicator function. If, at time \(s=t+u\) there is a change of phase from \(i\) to \(j\) then we need to map the orbit position \(\bs A(s^-)\) to an orbit position which corresponds to being at \(y_{k+1(i\in\mathcal N)} + c_iu\) but now in Phase~\(j\). Denote this mapping \(\bs D_{i,j}\) with the associated mark \((i,j)\).

Recall that, in general, at event epochs the orbit process may jump according to the matrix \(\bs D_m\) where \(m\) is the mark associated with the event. Specifically, if there is an event at time \(s\) with mark \(m\) the orbit will jump from \(\bs A(s^-)\) to \(\cfrac{\bs A(s^-)\bs D_m}{\bs A(s^-)\bs D_m\bs e}\).

The matrix \(\bs D_{i,j}\) must be chosen to ensure that the process constructed is a valid QBD-RAP. Formally, let \(\mathcal A\) be the convex hull of the set of all row-vectors of the form
\begin{align}
    \bs A(t) = \cfrac{\bs \alpha\left( \prod\limits_{n=1}^{N} e^{\bs S(t_{n}-t_{n-1})}\bs D_{i_n, i_{n+1}}\right)e^{\bs S (t-t_{N})}}{\bs \alpha\left( \prod\limits_{n=1}^{N} e^{\bs S(t_{i}-t_{i-1})}\bs D_{i_n, i_{n+1}}\right)e^{\bs S( t-t_{N})}\bs e}, \label{eqn: mathcal A set}
\end{align}
which sum to unity, where \(0\leq t_1\leq t_2\leq ... \leq t_N \leq t\), \(i_1,...,i_{N+1}\in\mathcal S\). We require that the mappings \(\bs D_{i,j}\) are such that the tuple \((\bs a,\bs S)\) is a valid representation of a matrix exponential distribution for all \(\bs a\) in \(\mathcal A\). This condition ensures that the approximation is a QBD-RAP and hence preserves all probabilistic properties. 

An additional convenient property is \(\bs D_{i,j} \bs e = \bs e\) as this means that the rate at which a change of phase from \(i\) to \(j\) of the QBD-RAP occurs is equal to 
\[\bs A(t+u) \bs D_{i,j}\bs e T_{ij} = T_{ij},\]
and therefore the distribution of time until a change from Phase~\(i\) to \(j\) is exponential. With this property we can show that the distribution of the phase process of the fluid queue and the distribution of the phase process of the QBD-RAP are the same.

Suppose that the first event after entering level \(k\) is a change of phase from \(i\in\mathcal P\) to \(j\in\mathcal P\) at time \(s\). From the discussion above, the orbit position which corresponds to the level process of the fluid queue being at \(y_k+x\) in Phase~\(j\in\mathcal P\) is 
\[\cfrac{\bs \alpha e^{\bs S|c_j| (x/|c_j|)}}{\bs \alpha e^{\bs S|c_j| (x/|c_j|)}\bs e} = \cfrac{\bs \alpha e^{\bs Sx}}{\bs \alpha e^{\bs S x}\bs e},\] 
which is independent of \(j\). Hence, upon a change of phase from \(i\in\mathcal P\) to \(j\in\mathcal P\), the orbit position does not need to change, i.e., \(\bs D_{i,j}=\bs I\), \(i,j\in\mathcal P\). 

By analogous arguments, the same choice \(\bs D_{i,j}=\bs I\) is reasonable for changes of phase from \(i\in\mathcal N\) to \(j\in\mathcal N\). 

Now suppose that the first event after entering level \(k\) is a change of phase from \(i\in\mathcal P\) to \(j\in\mathcal N\) at time \(s\). The orbit position immediately before \(s\) is \(\bs A(s^-)=\cfrac{\bs \alpha e^{\bs Sx}}{\bs \alpha e^{\bs Sx}\bs e}\) where \(x=|c_i|u\). We need to find a matrix \(\bs D_{i,j}\) such that the orbit position immediately after the jump, which is given by
\[
    \cfrac{\cfrac{\bs \alpha e^{\bs Sx}}{\bs \alpha e^{\bs Sx}\bs e}\bs D_{i,j}}{\cfrac{\bs \alpha e^{\bs Sx}}{\bs \alpha e^{\bs Sx}\bs e}\bs D_{i,j}\bs e} = \cfrac{\bs \alpha e^{\bs Sx}}{\bs \alpha e^{\bs Sx}\bs e}\bs D_{i,j},
\] 
corresponds to the fluid level being a distance of \(|c_i|u\) from \(y_k\) in Phase~\(j\in\mathcal N\), in some sense. 

Given the interpretation assigned to the orbit position, we would like our map to be such that 
\begin{align}
    \cfrac{\bs \alpha e^{\bs Sx}}{\bs \alpha e^{\bs Sx}\bs e}\bs D_{i,j} = \cfrac{\bs \alpha e^{\bs S(\Delta -x)}}{\bs \alpha e^{\bs S(\Delta -x)}\bs e},\label{eqn: orbit jump}
\end{align}
so that the orbit position for \(i\in \mathcal P\) maps directly to the corresponding orbit position for \(j\in\mathcal N\). 
However, we have had no success in finding such a matrix which also has the desired properties described above. Instead, we use
\begin{align}
	\bs D_{i,j} = \int_{y=0}^{\Delta-\varepsilon} e^{\bs Sy}\bs s \cfrac{\bs \alpha e^{\bs Sy}}{\bs \alpha e^{\bs Sy}\bs e}\wrt y + \int_{y=\Delta-\varepsilon}^\infty e^{\bs Sy}\bs s \cfrac{\bs \alpha e^{\bs S(\Delta-\varepsilon)}}{\bs\alpha e^{\bs S (\Delta-\varepsilon)}\bs e}\wrt y, \label{eqn: D instead}
\end{align}
which can be seen as an approximation to \eqref{eqn: orbit jump}. To see this, apply \(\bs D_{i,j}\) to \(\bs A(s^-)\) 
\begin{align*}
	\bs A(s^-)\bs D_{i,j} &= \cfrac{\bs \alpha e^{\bs Sx}}{\bs \alpha e^{\bs Sx}\bs e}\int_{y=0}^{\Delta-\varepsilon} e^{\bs Sy}\bs s \cfrac{\bs \alpha e^{\bs Sy}}{\bs \alpha e^{\bs Sy}\bs e}\wrt y + \cfrac{\bs \alpha e^{\bs Sx}}{\bs \alpha e^{\bs Sx}\bs e}\int_{y=\Delta-\varepsilon}^\infty e^{\bs Sy}\bs s \cfrac{\bs \alpha e^{\bs S(\Delta-\varepsilon)}}{\bs\alpha e^{\bs S (\Delta-\varepsilon)}\bs e}\wrt y.
\end{align*}
For \(2\varepsilon<x<\Delta-\varepsilon\), the first term is 
\[
	\int_{y=0}^{\Delta-\varepsilon} \cfrac{\bs \alpha e^{\bs Sx}}{\bs \alpha e^{\bs Sx}\bs e}e^{\bs Sy}\bs s \cfrac{\bs \alpha e^{\bs Sy}}{\bs \alpha e^{\bs Sy}\bs e}\wrt y \approx
	\int_{y=0}^{\Delta-\varepsilon} \delta(y-(\Delta-x)) \cfrac{\bs \alpha e^{\bs Sy}}{\bs \alpha e^{\bs Sy}\bs e}\wrt y=\cfrac{\bs \alpha e^{\bs S(\Delta-x)}}{\bs \alpha e^{\bs S(\Delta-x)}\bs e}.
\]
The second term is 
\begin{align*}
	\cfrac{\bs \alpha e^{\bs Sx}}{\bs \alpha e^{\bs Sx}\bs e}\int_{y=\Delta-\varepsilon}^\infty e^{\bs Sy}\bs s \cfrac{\bs \alpha e^{\bs S(\Delta-\varepsilon)}}{\bs\alpha e^{\bs S (\Delta-\varepsilon)}\bs e}\wrt y
	&=\cfrac{\bs \alpha e^{\bs Sx}}{\bs \alpha e^{\bs Sx}\bs e} e^{\bs S(\Delta-\varepsilon)}\bs e \cfrac{\bs \alpha e^{\bs S(\Delta-\varepsilon)}}{\bs\alpha e^{\bs S (\Delta-\varepsilon)}\bs e}
	\\&= \cfrac{\mathbb P(Z>x+\Delta-\varepsilon)}{\mathbb P(Z>x)}\cfrac{\bs \alpha e^{\bs S(\Delta-\varepsilon)}}{\mathbb P(Z>\Delta-\varepsilon)}.
\end{align*}
For any bounded integrable function \(g:[0,\infty)\to\mathbb R,\, |g(x)|\leq G\), 
\begin{align*}
	&\left|\int_{x=0}^\infty \cfrac{\mathbb P(Z>x+\Delta-\varepsilon)}{\mathbb P(Z>x)}\cfrac{\bs \alpha e^{\bs S(\Delta-\varepsilon)}}{\mathbb P(Z>\Delta-\varepsilon)}e^{\bs S x}\bs s g(x)\wrt x \right|
	\\&\leq\left|\int_{x=0}^\infty \cfrac{\var(Z)/\varepsilon^2}{\left(1-\var(Z)/\varepsilon^2\right)^2} \bs \alpha e^{\bs S(\Delta-\varepsilon)}e^{\bs S x}\bs s g(x)\wrt x \right|
	\\&\leq \cfrac{\var(Z)/\varepsilon^2}{\left(1-\var(Z)/\varepsilon^2\right)^2} G
\end{align*}
by Chebyshev's inequality, which is small when \(\var(Z)\) is small. In this sense, the second term is negligible, so \(\bs A(s^-)\bs D_{i,j}\approx \cfrac{\bs \alpha e^{\bs S(\Delta-x)}}{\bs \alpha e^{\bs S(\Delta-x)}\bs e}\). 

The matrix \(\bs D_{i,j}\) has the property that \((\bs a\bs D_{i,j},\bs S)\) is a valid representation of a matrix exponential distribution provided that \((\bs a, \bs S)\) is a valid representation of a matrix exponential distribution. By induction on \(N\) in Equation~\eqref{eqn: mathcal A set}, it can be shown that this property is a sufficient condition for \((\bs a, \bs S)\) to be valid matrix exponential representations for all \(\bs a \in \mathcal A\).  Moreover, it can also be shown that \(\bs D_{i,j}\) has the desired row-sum property.

Analogous arguments can be used to argue the same choice of \(\bs D_{i,j}\) is appropriate when the phase jumps from \(i\in\mathcal N\) to \(j\in\mathcal P\). Since \(\bs D_{i,j}\) in (\ref{eqn: D instead}) does not depend on \(i\) and \(j\), define \(\bs D=\bs D_{i,j}\) for all \(i,\,j\in\mathcal S\), \(\sign(c_i)\neq \sign(c_j)\). 

We have considered only the first change of phase from Phase~\(i\) to \(j\) with \(\sign(c_i)\neq \sign(c_j)\). Upon subsequent changes of phase between states with rates of opposite sign we claim that a reasonable choice for the jump matrix of the orbit process is again \(\bs D\). If a different jump matrix is specified the jump matrix depends on the number of transitions from states with rates of opposite sign, which enlarges the state space of the QBD-RAP as the state needs to contain information about the number of transitions from states with rates of opposite sign, and therefore makes the approximation more complex and less practical.

\subsubsection{On a change of level}\label{subsubsec: change of level}
On the event that there is a change of level at time \(s\) and given the phase is \(i\in\mathcal P\) the orbit is set to \(\bs A(s)=\bs \alpha\) as this orbit value corresponds to the level process \(\{X(t)\}_{t\geq0}\) being at \(y_{k+1}\) when the level is \(L(s)=k+1\). Similarly, on the event that there is a change of level at time \(s\) and given the phase is \(i\in\mathcal N\) the orbit is set to \(\bs A(s)=\bs \alpha\) as this orbit value corresponds to the level process \(\{X(t)\}_{t\geq0}\) being at \(y_{k}\) when the level is \(L(s)=k-1\).

\subsection{Boundary conditions}
Often boundary conditions are imposed on fluid queues directly. Other times, to enable computation for approximations, the fluid queue is truncated and boundary conditions are imposed. Here, we describe how to approximate regulated boundary conditions only, as this simplifies notation. More generally, it is possible to consider other boundary conditions \cite{a2023}. 

With the construction above, the event that the fluid level \(\{X(t)\}\) hits the boundary at \(0\) is approximated by the event that the QBD-RAP is in Level \(1\) in a Phase~\(i\in\mathcal N\) and there is a change of level. Upon hitting the boundary, the phase dynamics of the fluid queue can be captured without approximation. At the lower boundary the phase dynamics of the fluid queue are those of a transient CTMC with sub-generator \(\bs T_{--}\). Upon leaving the boundary the phase transitions according to \(\bs T_{-+}\) and the orbit process for the approximation is set to \(\bs \alpha\).

Analogous dynamics describe the approximation at the upper boundary.

\subsection{QBD-RAP generator}
The QBD-RAP construction can be summarised by writing down its generator matrix which is as follows.
\begin{align}
	\bs B &= \left[\begin{array}{ccccccc}
	\check{\bs{B}}_{0} & \check{\bs{B}}_{+} & & &&&\\
    \check{\bs{B}}_{-} & {\bs{B}}_{0} & {\bs{B}}_{+} &&&& \\
    & {\bs{B}}_{-} & {\bs{B}}_{0} & {\bs{B}}_{+}&&& \\
    & & \ddots & \ddots &\ddots &&\\
    &&& {\bs{B}}_{-} & {\bs{B}}_{0} & {\bs{B}}_{+}& \\
    &&&& {\bs{B}}_{-} & {\bs{B}}_{0} & \hat{\bs{B}}_{+}\\
    &&&&& \hat{\bs{B}}_{-} & \hat{\bs{B}}_{0} \\
	\end{array}\right],\label{eqn: BBBBbndry}
\end{align}
where 
\begin{align}
	\check{\bs B}_0 &= \bs T_{--},\quad
	\bs B_0 = \left[\begin{array}{cc} \bs B_{++} & \bs B_{+-} \\ \bs B_{-+} & \bs B_{--}\end{array}\right],\quad
    \hat{\bs B}_0 = \bs T_{++} \label{eqn: BBBB2a},
\end{align}
\begin{align}
	\check{\bs{B}}_{-} &= \left[\begin{array}{c}
		\bs 0 \\
		\bs{C}_-\otimes \bs s \\ 
		\end{array}\right],\quad
		\bs{B}_{-} = \left[\begin{array}{cc}
			\bs 0 & \bs 0 \\
			\bs 0 & \bs{C}_-\otimes (\bs s\bs\alpha) 
		\end{array}\right],
		\quad
        \hat{\bs{B}}_{-} = \left[\begin{array}{cc}
			\bs 0 & \bs T_{+-} \otimes\bs \alpha 
		\end{array}\right],
\end{align}
\begin{align}
	\check{\bs{B}}_{+} &= \left[\begin{array}{cc}
		\bs T_{-+} \otimes\bs \alpha & \bs 0
    \end{array}\right],\quad
	\bs{B}_{+} = \left[\begin{array}{cc}
		\bs{C}_+\otimes (\bs s\bs\alpha) & \bs 0 \\
        \bs 0 & \bs 0 
    \end{array}\right],\quad
	\label{eqn: BBBB2c}
    \hat{\bs{B}}_{+} = \left[\begin{array}{c}
        \bs{C}_+\otimes \bs s \\
		\bs 0 
		\end{array}\right],
\end{align}
and
\begin{align}
	\bs{B}_{++}  &= \bs{C}_+\otimes \bs{S}  + \bs{T}_{++}\otimes \bs{I},\quad
	\bs{B}_{+-}  =  \bs{T}_{+-}\otimes \bs{D}  ,
	\\\bs{B}_{-+}  &=  \bs{T}_{-+}\otimes \bs{D} ,\quad
	\bs{B}_{--}  = \bs{C}_-\otimes \bs{S}  + \bs{T}_{--}\otimes \bs{I}.\label{eqn: BBBB3}
\end{align}
The tri-diagonal structure arises from the skip-free nature of the level process.

The domain of \(\bs B\) is the set of vectors in \(\mathbb R^{(K+1)\times|S|\times p}\) where \(p\) is the order of the distribution of \(Z\). If we also attach the supremum norm, \(||\cdot||_{\infty}\), to this space, then we have a Banach space. Denote by \(\bs T(t)=e^{\bs B t}\) the transition semigroup of the QBD-RAP.

\subsection{Initial conditions}
Let \(\bs e_{k,i,n}\) denote the vector of zeros except a one in position 
\[n + (i-1)p - 1(k=0)|\mathcal P|p + 1(k>0)|\mathcal N|p+k|\mathcal S|p,\]
where 
\((k,i,n)\in\{(0,i,1)|i\in\mathcal N\}\cup \{(k,i,n)|k=1,...,K,i\in\mathcal S, n=1,...,p\}\cup \{(K+1,i,1)|i\in\mathcal P\}.\) The position of the one in the vector \(\bs e_{k,i,n}\) corresponds to basis function \(n\) in Phase \(i\) and Level \(k\) of the approximation. We approximate the initial condition \(X(0)=x_0,\,\rapphase(0)=i\) with the initial vector \(\bs e_{k,i,n}\tr{}\) for the QBD-RAP, where \(k,n\) are determined by \(X(0)=x_0,\,\rapphase(0)=i\) as follows.
\begin{itemize}
	\item If \(x_0=0\), \(i\in\mathcal N\), then \(k=0,n=1\),
	\item if \(x_0=b\), \(i\in\mathcal P\), then \(k=K+1,n=1\), 
	\item if \(x_0\in[y_\ell+t_m, y_\ell+t_{m+1}),\, j\in\mathcal P\), then \(k=\ell\) and \(n=m\),
	\item if \(x_0\in(y_{\ell+1}-t_{m+1}, y_{\ell+1}-t_{m}],\, j\in\mathcal N\), then \(k=\ell\) and \(n=m\).
\end{itemize}

The specification of the indices \(k\) and \(i\) in \(\bs e_{k,i,n}\) is obvious as there is a clear mapping between the level and phase of the fluid queue and QBD-RAP approximation. The specification of the index \(n\) relates to the particular choice of representation \((\bs \alpha, \bs S)\) of the ME used in the construction. Recall that we choose the representation such that \(\bs e_n\tr{} e^{\bs S t} \bs s=\mathbb P(Z-t_n>t \mid Z\geq t_n)\) and hence 
\[
	\cfrac{\bs \alpha e^{\bs S t_n}}{\bs \alpha e^{\bs S t_n}\bs e} =\bs e_n.
\]
Hence, the vector \(\bs e_n\) corresponds to the orbit position \(y_{k} + t_n\) for \(i\in\mathcal P\) and \(y_{k+1} - t_n\) for \(i\in\mathcal N\) and the initial condition above simply maps \(x_0\) to a nearby \(y_{k} + t_n\) if \(i\in\mathcal P\), or  \(y_{k+1} - t_n\) if \(i\in\mathcal N\), then uses the initial vector of the corresponding orbit at that position, \(\bs e_n\). 

%% file: closingoperator.tex
\section{The closing operator}\label{sec:closing operator}
The level process \(\{L(t)\}\) of the QBD-RAP approximates the process 
\begin{align}
	\displaystyle \left\{\sum_{k\in\{0,...,K+1\}} k1(X(t)\in\calD_{k,\varphi(t)})\right\}_{t\geq 0}. \label{eqn: kfafj}
\end{align}
This may be of interest in its own right. However, for some applications, this approximation may be too coarse---the level process tells us nothing about where in the intervals \(\calD_{k,i}\) the fluid level is. The purpose of this section is to introduce the idea of \emph{closing operators}, which are used to extract intra-level approximations to the distribution of \(X(t)\) using the value of the orbit \(\bs A(t)\). The intra-level approximation obtained in this way does not affect the dynamics of the QBD-RAP itself; we only take the information contained in \(\mathbb E[\bs A(t)1(L(t)=\ell,\rapphase(t)=i)]\) and write it in a descriptive way.

Suppose that at time \(t\) the QBD-RAP is in level \(L(t)=\ell\), Phase~\(\rapphase(t)=j\in\mathcal P\), and the orbit is \(\bs A(t)=\bs a\). If the QBD-RAP remains in Phase~\(j\), the QBD-RAP approximation will transition out of Level \(\ell\) in the infinitesimal time interval \([t+u,t+u+\wrt u]\) with probability \(\bs a e^{|c_j|\bs{S}u}|c_j|\bs s\wrt u\). At the time of the change of level we estimate the position of \(X(t+u)\) by \(X(t+u)\approx y_{\ell+1}\). Tracing this back to time \(t\), we estimate the position of \(X(t)\) as \(X(t)\approx y_{\ell+1} - |c_j|u\). Hence, the approximation to the probability of \(X(t) \in\wrt x\) is \(\bs a e^{|c_j|\bs S(y_{\ell+1}-x)/|c_j|}|c_j|\bs s \wrt x/|c_j|\), since \(\wrt x = |c_j|\wrt u\) where \(\wrt x\) is an infinitesimal in space and \(\wrt u\) is an infinitesimal with respect to time. An analogous intuition applies to \(\rapphase(t)=j\in\mathcal N\). 

The reasoning above leads to an approximation of the distribution of the fluid at time \(t\), 
\[\mathbb P(\bs X(t)\in{}(E,j) \mid \bs X(0)=(x_0,i)),\]
for measurable sets \(E\subseteq\mathcal D_{\ell,j}\), as
\begin{align}
		&\int_{x\in E}\int_{\bs a \in\mathcal A}\mathbb P((L(t),\bs A(t),\rapphase(t)) \in{}(\ell, \wrt \bs a, j)\mid (L(0),\bs A(0),\rapphase(0))=\bs y_0)\bs a e^{\bs{S}z_{\ell,j}(x)} \bs s \wrt x\nonumber\\
        &=\mathbb E[\bs A(t)1(L(t)=\ell,\rapphase(t)=j)] \int_{x\in E}e^{\bs{S}z_{\ell,j}(x)} \bs s \wrt x,
		\label{eqn: density approx aug}
\end{align}
where \(x_0\in\mathcal D_{\ell_0,i}\), \(\bs y_0 = (\ell_0,\bs e_{n}\tr{},i)\), where \(n\in\{1,...,p\}\) is such that \(x_0\in[y_\ell+t_n, y_\ell+t_{n+1})\) if \(i\in\mathcal P\) or \(x_0\in(y_{\ell+1}-t_{n+1}, y_{\ell+1}-t_{n}]\) if \(i\in\mathcal N\), and 
\[
	z_{\ell,j}(x) = \begin{cases} 
		x - y_\ell, & j\in\mathcal N, \\
		y_{\ell+1} - x,& j\in\mathcal P.
	\end{cases}
\]
The estimate (\ref{eqn: density approx aug}) is defective -- it does not integrate to 1 for any finite-order matrix exponential distribution. 

To rectify the defectiveness we instead approximate \(\mathbb P(\bs X(t)\in{}(\wrt x, j)\mid \bs X(0)=(x_0, i)) \) by
\begin{align}
		&\displaystyle\int_{\bs a \in\mathcal A}\mathbb P((L(t),\bs A(t),\rapphase(t)) \in{}(\ell, \wrt \bs a,j)\mid (L(0),\bs A(0),\rapphase(0)) = \bs y_0)\bs a \bs v_{\ell,j}(x) \wrt x \nonumber\\
    &=\mathbb E[\bs A(t)1(L(t)=\ell,\rapphase(t)=j)] \int_{x\in E}\bs v_{\ell,j}(x) \wrt x\label{eqn: density approx 4+}
\end{align}
{where}
\begin{align}
		 \bs v_{\ell,j}(x) = \begin{cases}
			 \left(e^{ \bs{S}(y_{\ell+1}-x)} + e^{ \bs{S}(2\Delta-(y_{\ell+1}-x))} \right)\left[I-e^{\bs S 2\Delta}\right]^{-1}\bs s, & j\in\mathcal P, \\ 
			 \left(e^{ \bs{S}(x-y_\ell)}+e^{ \bs{S}(2\Delta-(x-y_\ell))}\right) \left[I-e^{\bs S 2\Delta}\right]^{-1}\bs s, & j\in\mathcal N. \\ 
		\end{cases}
\end{align}
Here, we take the density function \(\bs a e^{\bs S x}\bs s\), defined on \(x\in[0,\infty)\), and map it to a density function on \(w\in[0,\Delta)\) by setting \(w=\Delta - |(x \,\mathrm{mod }\, 2\Delta )-\Delta|\). The approximation (\ref{eqn: density approx 4+}) leads to an approximation which converges and, due to the interpretation as probability density, ensures the approximation has probabilistic properties.

We now formulate the above in an operator-theoretic framework and define its properties. Let \(f\in\mathcal L\), \(f:[0,b]\times \mathcal S\to \mathbb R\). Define the \textit{closing operator} for the QBD-RAP scheme as the operator \(\bs V\) structured as the vector of operators
\begin{align*}
	\bs Vf&=\left[\begin{array}{l}\vligne{\bs V_{0,j}f}_{j\in\mathcal N} \\ \vligne{\bs V_{\ell,j}f}_{\ell=1,...,K,\,j\in\mathcal S} \\  \vligne{\bs V_{K+1,j}f}_{j\in\mathcal P} \end{array}\right],
\end{align*}
where 
\begin{align*}
	\bs V_{\ell,j}f=\int_{\mathcal D_{\ell,j}}\bs v_{\ell,j}(x)f(x,j)\wrt x,\, {\ell=1,...,K,\,j\in\mathcal S},
\end{align*}
\begin{align*}
	\bs V_{0,j}f=f(0,j),\,{j\in\mathcal N},\quad  \mbox{ and, }\quad
	\bs V_{K+1,j}f=f(b,j), \, {j\in\mathcal P}.
\end{align*}
The closing operator is important because it maps from \(\mathcal L\), the domain of the transition semigroup of the fluid queue, to \(\mathbb R^{(K+1)\times|S|\times p}\), the domain of the transition semigroup of the QBD-RAP.

\begin{remark}
	An alternative interpretation of the orbit process-closing vector construction is as a measure-valued process, \(\left\{\bs A(t)\bs V_{L(t),\phi(t)}\right\}\), a perspective that may connect the construction to the broader theory of measure-valued Markov processes; we do not pursue this further here.
\end{remark}

Observe that \(\bs V\) is a linear operator, \(\bs V(f+h)=\bs Vf + \bs Vh\), and is bounded, \(||\bs Vf||_{\infty} \leq ||f||_{\infty}\), since, if \(||f||_{\infty}= F\) then 
\[\displaystyle||\bs V_{\ell,j}f||_{\infty}=\max_{n=1,...,p}\left|\int_{\mathcal D_{\ell,j}}\bs e_n\tr{}\bs v_{\ell,j}(x)f(x,j)\wrt x\right|\leq \max_{n=1,...,p}\int_{\mathcal D_{\ell,j}}\left|\bs e_n\tr{}\bs v_{\ell,j}(x)\right|F\wrt x=F.\]
Hence, \(\bs V\) is a bounded linear operator from \(\mathcal L\to\mathbb R^{(K+1)\times|S|\times p}\). Using the notation of semigroups the approximation of \(T(t)f\) is \(\bs T(t)\bs Vf\).

From the definition of the closing operator, for \(j\in\mathcal P\),
\begin{align}
	\bs V_{\ell,j}f &= \int_{\mathcal D_{\ell,j}}\bs v_{\ell,j}(x)f(x,j)\wrt x \nonumber
	%
	%
	\\&=\displaystyle \int_{\mathcal D_{\ell,j}}e^{ \bs{S}(y_{\ell+1}-x)}\left[I-e^{\bs S 2\Delta}\right]^{-1}\bs sf(x,j)\wrt x + \int_{\mathcal D_{\ell,j}}e^{ \bs{S}(2\Delta-(y_{\ell+1}-x))} \left[I-e^{\bs S 2\Delta}\right]^{-1}\bs sf(x,j)\wrt x\nonumber 
	\\&=\displaystyle \sum_{k=0}^\infty\int_{y_\ell}^{y_{\ell+1}} e^{ \bs{S}(k2\Delta + y_{\ell+1}-x)} \bs sf(x,j)\wrt x + \sum_{k=0}^\infty \int_{y_\ell}^{y_{\ell+1}} e^{ \bs{S}((k+1)2\Delta-(y_{\ell+1}-x))} \bs sf(x,j)\wrt x,\label{eqn: geom}
\end{align}
where the last equality is due to the geometric series identity, \(\left[I-e^{\bs S 2\Delta}\right]^{-1} = \sum\limits_{k=0}^\infty  e^{ \bs{S}k2\Delta}\) and hence, \(e^{ \bs{S}x}\left[I-e^{\bs S 2\Delta}\right]^{-1} = e^{ \bs{S}x}\sum\limits_{k=0}^\infty  e^{ \bs{S}k2\Delta}=\sum\limits_{k=0}^\infty  e^{\bs{S}(x + k2\Delta)}\). By change of variables \eqref{eqn: geom} is equal to 
\begin{align*}
	&\displaystyle \sum_{k=0}^\infty  \int_{k2\Delta+\Delta}^{k2\Delta}e^{ \bs{S}x} \bs sf(k2\Delta+y_{\ell+1}-x,j)(-1)\wrt x 
 + \sum_{k=0}^\infty\int_{k2\Delta+\Delta}^{(k+1)2\Delta} e^{ \bs{S}x} \bs sf(x-(k+1)2\Delta+y_{\ell+1},j)\wrt x
	\\&=\displaystyle \sum_{k=0}^\infty  \int_{k2\Delta}^{(k+1)2\Delta} e^{ \bs{S}x} \bs sf(y_{\ell} + |x \,\mathrm{mod }\, 2\Delta - \Delta|,j)\wrt x
	\\&= \displaystyle\int_{[0,\infty)} e^{ \bs{S}x} \bs sf(y_{\ell} + |x \,\mathrm{mod }\, 2\Delta - \Delta|,j)\wrt x.
\end{align*}
Similarly, for \(j\in\mathcal N\) it can be shown that 
\begin{align*}
	\bs V_{\ell,j}f
	= \displaystyle\int_{[0,\infty)} e^{ \bs{S}x} \bs sf(y_{\ell+1} - |x \,\mathrm{mod }\, 2\Delta - \Delta|,j)\wrt x.
\end{align*}
Hence,
\begin{align*}
	&\bs e_n\tr{}\bs V_{\ell,j}f
	%
	= \begin{cases}
		\mathbb E\left[f(y_{\ell} + |(Z-t_n \,\mathrm{mod }\, 2\Delta )-\Delta|,j)\mid Z\geq t_n\right], & j\in\mathcal P,
		\\\mathbb E\left[f(y_{\ell+1} - |(Z-t_n \,\mathrm{mod }\, 2\Delta )-\Delta|,j)\mid Z\geq t_n\right], & j\in\mathcal N.
	\end{cases}
\end{align*}
Intuitively, assuming that the variance of \(Z\) is small and therefore \(Z\approx \Delta\), then 
\begin{align*}
	\bs e_n\tr{}\bs V_{\ell,j}f
	&\approx \begin{cases}
		f(y_{\ell} + t_n,j), & j\in\mathcal P,
		\\f(y_{\ell+1} - t_n,j), & j\in\mathcal N.
	\end{cases}
\end{align*}
With this intuition \(\bs Vf\) defines a simple function, denoted by \(f_p(x,j)=\bs Vf(x,j)\), which approximates \(f\),
\begin{align*}
	f_p(x, j) &
	\approx \begin{cases}
		f(0,j), & j\in\mathcal N, x=0,
		\\f(y_{\ell}+t_n,j), & j\in\mathcal P,\,x\in \left[y_{\ell}+t_n,y_{\ell}+t_{n+1}\right),\, n=1,...,p,
		\\ f(y_{\ell+1}-t_n,j),& j\in\mathcal N,\,x\in \left(y_{\ell+1}-t_{n+1},y_{\ell+1}-t_{n}\right],\, n=1,...,p,
		\\f(b,j), & j\in\mathcal P, x=b.
	\end{cases}
\end{align*}
In the appendix, we establish bounds for this approximation and other related bounds, which we later use to prove convergence of the QBD-RAP to the fluid queue. One of the main consequences of the results in the appendix is the following.
\begin{corollary}\label{eqn: fns converge}
	Let \(f:[0,b]\times \mathcal S \to \mathbb R\) be bounded and Lipschitz continuous, \(|f(x,j)|\leq G\) and \(|f(x,j)-f(y,j)|\leq L|x-y|\). Further, let \(\bs V^{(p)},\,p=1,2,...\) be a sequence of closing operators constructed from matrix exponential random variables \(Z^{(p)}\) with mean \(\Delta\) and variance \(Var(Z^{(p)})\to 0\), and define \(f_p(x,j)=\bs V^{(p)}f(x,j)\), then 
	\begin{align*}
		f_p(x, j) \to f(x,j)
	\end{align*}
	uniformly as \(p\to\infty\). 
\end{corollary}
\begin{proof}
	The result is a consequence of the bound in Corollary~\ref{cor: Vf approximates f} in the appendix, upon the construction of a sequence of matrix exponential random variables \(\{Z^{(p)}\}_{p}\) (and the required representations of those random variables) with mean \(\Delta\) and variance \(Var(Z^{(p)})\to 0\).
\end{proof}

%% file: convergence.tex
\section{Distributional results}
\label{sec:Convergence of the QBD-RAP approximation}
In this section we provide convergence and distributional results about the QBD-RAP approximation. First we show that, under certain conditions, the phase process of the QBD-RAP and the fluid queue have the same distribution. Next, our main result shows a weak convergence of the distribution of the QBD-RAP approximation to the distribution of the fluid queue, via convergence of generators.

\subsection{The distribution of the phase process}
Let \(\tau_{0}=0\) and denote the time of the \(n\)th level change of the QBD-RAP by 
\[
    \tau_n=\inf\left\{t>\tau_{n-1}|L(t)\neq L(\tau_{n-1})\right\}.
\]

The following result states that the phase process of the QBD-RAP, \(\left\{{\rapphase}(t)\right\}\), has the same distribution as the phase process of the fluid queue, \(\left\{{\varphi}(t)\right\}\), when the level processes are unbounded. The condition that the level processes are unbounded can be relaxed to the condition that the phase does not depend on the level process (e.g., a regulated boundary), however the notation in the proof is simpler in the unbounded case.
\begin{theorem}\label{thm: 1}
	Let \(\Theta_k,\,k=1,2,...\), be the time of the \(k\)th jump of the phase process, \(\left\{{\rapphase}(t)\right\}\), of an (unbounded) QBD-RAP approximation to a fluid queue. For any valid initial orbit \(\bs a\in\mathcal A\) and phase \(\rapphase(0)=i\), \(\Theta_k-\Theta_{k-1}\) is exponentially distributed with rate \(|T_{jj}|\), \(j=\rapphase(\Theta_{k-1})\) and, at time \(\Theta_k\), the phase jumps to \(i\neq j\) with probability \(T_{ij}/|T_{jj}|\). Hence, \(\{\rapphase(t)\}\) and \(\{\varphi(t)\}\) have the same probability law, where \(\left\{{\varphi}(t)\right\}\) is the phase process of the fluid queue. 
\end{theorem}
\begin{proof}
	Consider the sequence of hitting times \(\{\tau_n\}_{n\geq 0}\) where \(\tau_0=0\) and \(\tau_n\) is the time of the \(n\)-th change of level of the QBD-RAP and suppose we partition \(\left\{\Theta_1> t \right\}\) with respect to \(\left\{\tau_{n-1}< t \leq\tau_n\right\},\) \(n=1,2,\dots\). Let \(t_0=0\) and \(U_n=\left\{\Theta_1> t \right\}\cap \left\{\tau_{n-1}< t \leq\tau_n\right\},\) \(n=1,2,...\). Using Equation~(\ref{eqn: rapproperty}) and Theorem~\ref{thm:qbdrapchar}, by partitioning on the times of the first \(n-1\) level changes, \(\tau_1,\dots,\tau_{n-1}\), we get 
	\begin{align}
		\nonumber &\mathbb P(U_n \mid \bs A(0) = \bs a, \rapphase(0)=i) 
		\\\nonumber &= \int_{t_1=t_0}^t\hdots\int_{t_{n-1}=t_{n-2}}^t \mathbb P(U_n, \tau_{n-1}\in{}\wrt t_{n-1},\dots, \tau_1\in\wrt t_1 \mid \bs A(0) = \bs a, \rapphase(0)=i)
		\\\nonumber &=\displaystyle{\int_{t_1=t_0}^t\hdots\int_{t_{n-1}=t_{n-2}}^t} \bs a e^{(|c_i|\bs{S}+T_{ii}\bs{I}_p)t_1}|c_i|\bs{s}
			\left(\prod_{k=2}^{n-1} \bs \alpha e^{(|c_i|\bs{S}+T_{ii}\bs{I}_p)(t_k-t_{k-1})} |c_i|\bs{s}\right) 
			\\&\quad\displaystyle{\bs \alpha e^{(|c_i|\bs{S}+T_{ii}\bs{I}_p)(t-t_{n-1})}\bs e}\wrt t_{n-1}\wrt t_{n-2}\hdots\wrt t_{1}. \label{eqn:lamnvczb}
	\end{align}
	{Since \(T_{ii}\bs{I}_p\) commutes with \(|c_i|\bs{S}\), \(e^{T_{ii}t_k},\,k=1,...,n-1\) are scalars, and since \(t_1+(t_2-t_1)+\hdots+(t_{n-1}-t_{n-2})+(t-t_{n-1})=t\), (\ref{eqn:lamnvczb}) is equal to}
	\begin{align*}
		& \displaystyle{\int_{t_1=0}^t\hdots\int_{t_{n-1}=t_{n-2}}^t} \bs a e^{|c_i|\bs{S}t_1}|c_i|\bs{s} 
			\left(\prod_{k=2}^{n-1} \bs \alpha e^{|c_i|\bs{S}(t_k-t_{k-1})} |c_i|\bs{s}\right) \bs \alpha e^{|c_i|\bs{S}(t-t_{n-1})}\bs e
			e^{T_{ii}t} \wrt t_{n-1}\hdots\wrt t_{1}
		\\&= \mathbb P(\tau_{n-1}< t \leq \tau_n\mid \bs A(0) = \bs a, \rapphase(0)=i)\displaystyle{  e^{T_{ii}t} }.
	\end{align*}
	Hence, by the law of total probability, 
	\begin{align}
		\mathbb P(\Theta_1> t\mid \bs A(0) = \bs a, \rapphase(0)=i)  \nonumber
		&= \sum_{n=1}^\infty \mathbb P(U_n\mid \bs A(0) = \bs a, \rapphase(0)=i)\nonumber
		\\&= \sum_{n=1}^\infty \mathbb P(\tau_{n-1} < t \leq \tau_n\mid \bs A(0) = \bs a, \rapphase(0)=i)\displaystyle{  e^{T_{ii}t} }\nonumber
		\\&= e^{T_{ii}t},\label{eqn: dist Theta_i}
	\end{align}
	and therefore \(\Theta_1\) has an exponential distribution with rate \(|T_{ii}|\). 
    
    Upon leaving state \(i\) at time \(\Theta_1\), \(\{\rapphase(t)\}\) transitions to state \(j\neq i\) with probability 
	\begin{align*}
            	\cfrac{\left(\cfrac{\bs A(t) \bs{D}^{1(\sign(c_i)\neq \sign(c_j))}T_{ij}\bs e}{\displaystyle\sum_{\substack{j\in\mathcal S\\j\neq i}}\bs A(t) \bs{D}^{1(\sign(c_i)\neq \sign(c_j))}T_{ij}\bs e + \bs A(t) |c_i|\bs{s}}\right)}
            	{\left(\cfrac{\displaystyle\sum_{\substack{j\in\mathcal S\\j\neq i}}\bs A(t) \bs{D}^{1(\sign(c_i)\neq \sign(c_j))}T_{ij}\bs e}{\displaystyle\sum_{\substack{j\in\mathcal S\\j\neq i}}\bs A(t) \bs{D}^{1(\sign(c_i)\neq \sign(c_j))}T_{ij}\bs e + \bs A(t) |c_i|\bs{s}}\right)}
            	&=\cfrac{\bs A(t) \bs eT_{ij}}
            	{\displaystyle\sum_{\substack{j\in\mathcal S\\j\neq i}}\bs A(t) \bs eT_{ij}}
	= \cfrac{T_{ij}}
            	{-T_{ii}}.
	 \end{align*}
	
	Now let \(N_L(t)\) be the counting process which counts the number of transitions of \(\{\rapphase(t)\}\) up to time \(t\). For any stopping time \(\sigma\) consider 
    \begin{align}
		&\mathbb P(\Theta_{N_L(\sigma)+1}-\sigma > s,\mid \{\bs A(u), \rapphase(u),\, u\in[0,\sigma]\}) \nonumber
    \end{align} 
    where \(\{\bs A(u), \rapphase(u),\, u\in[0,\sigma]\}\) is the entire history of the QBD-RAP up to the stopping time \(\sigma\). From the strong Markov property
    \begin{align}
		&\mathbb P(\Theta_{N_L(\sigma)+1}-\sigma > s,\mid \{\bs A(u), \rapphase(u),\, u\in[0,\sigma]\}) \nonumber
        =\mathbb P(\Theta_{N_L(\sigma)+1}-\sigma > s,\mid \bs A(\sigma), \rapphase(\sigma)), \nonumber
    \end{align}
    and since the QBD-RAP process is time-homogeneous 
    \begin{align}
		\mathbb P(\Theta_{N_L(\sigma)+1}-\sigma > s\mid \bs A(\sigma), \rapphase(\sigma)) \nonumber
        &=\mathbb P(\Theta_{1}> s\mid \bs A(0), \rapphase(0))
    \end{align}
    which is the left-hand side of Equation~(\ref{eqn: dist Theta_i}). Hence, using the same arguments, \(\Theta_{N_L(\sigma)}-\sigma\) given \(\{\bs A(u), \rapphase(u),\, u\in[0,\sigma]\}\) is also exponentially distributed with rate \(|T_{ii}|,\,i=\rapphase(\sigma)\) and upon a change of phase at time \(\Theta_{N_L(\sigma)+1}\) the new phase is \(j\) with probability \(\cfrac{T_{ij}}{-T_{ii}}\). This is precisely the law of the process \(\left\{\varphi(t)\right\}\) and hence we have shown that the processes \(\left\{{\rapphase}(t)\right\}\) and \(\left\{\varphi(t)\right\}\) have the same probability law.
\end{proof}

\subsection{Weak convergence}
In this section we prove convergence results of the QBD-RAP scheme to ultimately show that the scheme converges in a weak sense. Specifically, we show that the generator of the QBD-RAP converges to the generator of the fluid queue. By Theorem~6.1~in~\cite{ethierkurtz}, this implies that the transition semigroup of the QBD-RAP converges to the transition semigroup of the fluid queue, and hence that the QBD-RAP approximation converges in distribution to the fluid queue.

\begin{lemma}[The generators are close to each other at a collection of points]\label{lem: generators are close}
	Let \(f:[0,b]\times \mathcal S\to\mathbb R\) be any function with bounded and Lipschitz continuous derivative, \(|f_x(x,j)| \leq G'\) and \(|f_x(x,j)-f_x(y,j)|\leq L'|x-y|\). Then
	\[||\bs e_{k,i,n}\tr{}\bs B \bs V f - Bf(y_k + t_n,i)||_\infty \leq M(1+\bs e_n\tr{} \bs s)\left({\var(Z)^{1/3}}/\left(1-\var(Z)^{1/3}\right)\right),\quad i\in\mathcal P,\]
	and
	\[||\bs e_{k,i,n}\tr{}\bs B \bs V f - Bf(y_{k+1}-t_n,i)||_\infty \leq M(1+\bs e_n\tr{} \bs s)\left({\var(Z)^{1/3}}/\left(1-\var(Z)^{1/3}\right)\right),\quad i\in\mathcal N,\]
	where \(M<\infty\) depends on \(G'\) and \(L'\) but is independent of \(k,\,i,\) and \(n\), and \(\bs e_n\tr{}\bs s\geq 0\).
\end{lemma}
\begin{proof}
	Computing \(\bs e_{k,i,n}\tr{}\bs B \bs V f\) gives
	\begin{enumerate}[label=Case \arabic*:]
		\item for \(k=0,\, i\in\mathcal N,\, n=1,\)
		\begin{align}
		\bs e_{0,i,1}\tr{}\bs B \bs V f &= \displaystyle\sum_{j\in\mathcal N}T_{ij}\bs V_{0,j}f + \sum_{j\in\mathcal P}T_{ij}\bs\alpha\bs V_{1,j}f. \label{eqn: first}
		\end{align}
		\item for \(k=1,\,i\in\mathcal N, \, n=1,...,p,\)
		\begin{align}
		\bs e_{1,i,n}\tr{}\bs B \bs V f &=
			\displaystyle |c_i|\bs e_n\tr{}\bs s\bs V_{0,i}f + |c_i|\bs e_n\tr{}\bs S\bs V_{1,i}f + \sum_{j\in\mathcal P}T_{ij}\bs e_n\tr{}\bs D\bs V_{1,j}f + \sum_{j\in\mathcal N}T_{ij}\bs e_n\tr{}\bs V_{1,j}f.
		\end{align}
		\item for \(k=1,...,K-1,\, i\in\mathcal P, \, n=1,...,p,\)
		\begin{align}
		\bs e_{k,i,n}\tr{}\bs B \bs V f &=
			\displaystyle |c_i|\bs e_n\tr{}\bs s\bs \alpha\bs V_{k+1,i}f + |c_i|\bs e_n\tr{}\bs S\bs V_{k,i}f + \sum_{j\in\mathcal P}T_{ij}\bs e_n\tr{}\bs V_{k,j}f + \sum_{j\in\mathcal N}T_{ij}\bs e_n\tr{}\bs D\bs V_{k,j}f.
		\end{align}
		\item for \(k=2,...,K,\, i\in\mathcal N, \, n=1,...,p,\)
		\begin{align}
		\bs e_{k,i,n}\tr{}\bs B \bs V f &=
			\displaystyle |c_i|\bs e_n\tr{}\bs s\bs \alpha\bs V_{k-1,i}f + |c_i|\bs e_n\tr{}\bs S\bs V_{k,i}f + \sum_{j\in\mathcal P}T_{ij}\bs e_n\tr{}\bs D\bs V_{k,j}f + \sum_{j\in\mathcal N}T_{ij}\bs e_n\tr{}\bs V_{k,j}f.
		\end{align}
		\item for \(k=K,\, i\in\mathcal P, \, n=1,...,p,\)
		\begin{align}
		\bs e_{K,i,n}\tr{}\bs B \bs V f &=
			\displaystyle |c_i|\bs e_n\tr{}\bs s\bs V_{K+1,i}f + |c_i|\bs e_n\tr{}\bs S\bs V_{K,i}f + \sum_{j\in\mathcal P}T_{ij}\bs e_n\tr{}\bs V_{K,j}f + \sum_{j\in\mathcal N}T_{ij}\bs e_n\tr{}\bs D\bs V_{K,j}f.
		\end{align}
		\item for \(k=K+1,\, i\in\mathcal P,\,n=1,\)
		\begin{align}
		\bs e_{K+1,i,1}\tr{}\bs B \bs V f &=
			\displaystyle \sum_{j\in\mathcal P}T_{ij}\bs V_{K+1,j}f + \sum_{j\in\mathcal N}T_{ij}\bs\alpha\bs V_{K,j}f.\label{eqn: last}
		\end{align}
\end{enumerate}
	Now, we write each of the terms on the right-hand sides of \eqref{eqn: first}-\eqref{eqn: last} above in terms of \(f\) plus an error term, utilising the results in the appendix to provide the error terms.
	By definition, 
	\begin{align}\bs V_{0,j}f = f(0,j),\quad \bs V_{K+1,j}f = f(b,j). \label{eqn: first2}\end{align}
	Since \(f\) has a bounded domain and bounded and Lipschitz continuous derivatives, \(f\) is also bounded and Lipschitz continuous. Hence, we can apply the results from the appendix.
	By Lemma~\ref{lem: cv evaluates g},
	\begin{align}
		\bs e_n\tr{}\bs V_{k,j}f=\begin{cases}
		f(y_k+t_n,j) + r_{k,j,n}^1, &j\in\mathcal P,
		\\f(y_{k+1}-t_n,j) + r_{k,j,n}^1, &j\in\mathcal N,
	\end{cases}\end{align}
	with the special case \(\bs e_1\tr{}=\bs\alpha\), by Corollary~\ref{cor: S cv evaluates g'},
	\begin{align}
		\bs e_n\tr{}\bs S\bs V_{k,j}f = \begin{cases}
		-\bs e_n\tr{} \bs s f(y_{k+1},j)+ f_x(y_k + t_n,j)+ r{}'{}^{1}_{k,j,n}, & j\in\mathcal P,
		\\-\bs e_n\tr{} \bs s f(y_k,j)- f_x(y_{k+1}-t_n,j)+ r{}'{}^{1}_{k,j,n}, & j\in\mathcal N,
	\end{cases}
	\end{align}
	and by Lemma~\ref{lem: D cv evaluates g},
	\begin{align}
		\bs e_n\tr{}\bs D \bs V_{k,j}f  = \begin{cases}
			f(y_{k+1}-t_n, j) + r^4_{k,j,n}, & j\in\mathcal P,
			\\ f(y_{k}+t_n, j) + r^4_{k,j,n}, & j\in\mathcal N.
		\end{cases}  \label{eqn: last2}
	\end{align}
	Substituting the expressions in \eqref{eqn: first2}-\eqref{eqn: last2} into \(\bs e_{k,i,n}\tr{}\bs B \bs V f\) in \eqref{eqn: first}-\eqref{eqn: last} effectively proves the result. We show the details for the proof of Cases 1-3 only as Cases 4-6 are analogous. 
	\begin{enumerate}[label=Case \arabic*:]
	\item for \(k=0,\, i\in\mathcal N,\, n=1,\)
	\begin{align*}
		\bs e_{0,i,1}\tr{}\bs B \bs V f &=
		\displaystyle\sum_{j\in\mathcal N}T_{ij}f(0,j) + \sum_{j\in\mathcal P}T_{ij}f(0, j) +  \sum_{j\in\mathcal P}T_{ij}r^4_{1,j,1}
		\\&= B f(0, j) + r^5_{0,i,1},
	\end{align*}
	where \(\displaystyle r^5_{0,i,1} = \sum_{j\in\mathcal P}T_{ij}r^4_{1,j,1}.\)
	\item for \(k=1,\,i\in\mathcal N, \, n=1,...,p,\)
	\begin{align*}
		\bs e_{1,i,n}\tr{}\bs B \bs V f &= \displaystyle |c_i|\bs e_n\tr{}\bs sf(0,i) + |c_i|(-\bs e_n\tr{} \bs s f(0,i)- f_x(\Delta-t_n,i)+ r{}'{}^{1}_{1,i,n}) 
		\\&\quad{} + \sum_{j\in\mathcal P}T_{ij}(f(\Delta-t_n, j) + r^4_{1,j,n}) + \sum_{j\in\mathcal N}T_{ij}(f(\Delta-t_n) + r_{1,j,n}^1), 
		\\&= \displaystyle -|c_i|f_x(\Delta-t_n,i)
		+ \sum_{j\in\mathcal S}T_{ij}f(\Delta-t_n, j) + \sum_{j\in\mathcal N}T_{ij}f(\Delta-t_n) + r^6_{1,i,n}
		\\&= Bf(\Delta-t_n,i) + r^6_{1,i,n}, 
	\end{align*}
	where \(\displaystyle r^6_{1,i,n}=r{}'{}^{1}_{1,i,n}+\sum_{j\in\mathcal S}T_{ij}r^4_{1,j,n}+\sum_{j\in\mathcal N}T_{ij}r_{1,j,n}^1\).
	\item for \(k=1,...,K-1,\, i\in\mathcal P, \, n=1,...,p,\)
	\begin{align*}
		\bs e_{k,i,n}\tr{}\bs B \bs V f &= \displaystyle |c_i|\bs e_n\tr{}\bs s(f(y_{k+1},i) + r_{k+1,i,1}^1) + |c_i|(-\bs e_n\tr{} \bs s f(y_{k+1},i)+ f_x(y_k + t_n,i)+ r{}'{}^{1}_{k,i,n}) 
		\\&\quad{} + \sum_{j\in\mathcal P}T_{ij}(f(y_k+t_n,j) + r_{k,j,n}^1) + \sum_{j\in\mathcal N}T_{ij}(f(y_{k}+t_n, j) + r^4_{k,j,n})
		\\ &= \displaystyle |c_i|f_x(y_k + t_n,i) + \sum_{j\in\mathcal P}T_{ij}f(y_k+t_n,j) + \sum_{j\in\mathcal N}T_{ij}f(y_{k}+t_n, j) + r^7_{k,i,n}
		\\ &= \displaystyle Bf(y_k+t_n,j) + r^7_{k,i,n},
	\end{align*}
	where \(\displaystyle r^7_{k,i,n} = |c_i|\bs e_n\tr{}\bs s r_{k+1,i,1}^1 + |c_i|r{}'{}^{1}_{k,i,n}
	+ \sum_{j\in\mathcal P}T_{ij}r_{k,j,n}^1 + \sum_{j\in\mathcal N}T_{ij}r^4_{k,j,n}.\)
	\end{enumerate}
	From Lemma~\ref{lem: cv evaluates g}, Corollary~\ref{cor: S cv evaluates g'} and Lemma~\ref{lem: D cv evaluates g}, the terms \(|r_{k,i,n}^1|,\, |r{}'{}^{1}_{k,i,n}|,\) and \(|r^4_{k,j,n}|\) are less than or equal to \[M'\left({\var(Z)^{1/3}}/\left(1-\var(Z)^{1/3}\right)\right)\] and \(|\bs e_n\tr{}\bs sr_{k,i,n}^1|\) is less than or equal to \[\bs e_n\tr{}\bs s M'\left({\var(Z)^{1/3}}/\left(1-\var(Z)^{1/3}\right)\right)\] for some \(M'<\infty\) independent of \(k,i,n\) and \(\bs e_n\tr\bs s\geq 0\). Hence, the error terms \(|r^5_{0,i,1}|,\, |r^6_{1,i,n}|,\, |r^7_{k,i,n}|\) are less than or equal to \(M(1+\bs e_n\tr{} \bs s)\left({\var(Z)^{1/3}}/\left(1-\var(Z)^{1/3}\right)\right)\) for some \(M<\infty\) independent of \(k,i,n\). This completes the proof.
\end{proof}

Define the notation
\[\bs B\bs V f(x,i) = \begin{cases}
	\bs e_{0,i,1}\tr{} \bs B \bs V f, & x=0, i\in\mathcal N,
	\\\bs e_{K+1,i,1}\tr{} \bs B \bs V f, & x=b, i\in\mathcal P,
	\\\bs e_{k,i,n}\tr{} \bs B \bs V f, & x\in[y_k+t_n, y_k+t_{n+1}),\, i\in\mathcal P,
	\\\bs e_{k,i,n}\tr{} \bs B \bs V f, & x\in(y_{k+1}-t_{n+1}, y_{k+1}-t_{n}],\, i\in\mathcal N,
\end{cases}\]
and let 
\[
	\bar x_{i}(x) = \begin{cases}
		0, & x=0,\\
		y_k+t_n, & x\in[y_k+t_n, y_k+t_{n+1}),\, i\in\mathcal P,\\
		y_{k+1}-t_n, & x\in(y_{k+1}-t_{n+1}, y_{k+1}-t_{n}],\, i\in\mathcal N,\\
		b, & x=b.
	\end{cases}
\]

\begin{theorem}[The generators are close to each other for functions in a core]\label{thm: generators converge}
	Let \(\mathcal C\) be the set of functions \(f:[0,b]\times \mathcal S\to\mathbb R\) with bounded and Lipschitz continuous derivative, \(|f_x(x,j)| \leq G'\) and \(|f_x(x,j)-f_x(y,j)|\leq L'|x-y|\) (\(G'\) and \(L'\) may depend on \(f\)). Then, for \(f\in\mathcal C\) and any \(\delta>0\), there exists a matrix exponential random variable \(Z\) with \(\var(Z)\) sufficiently small such that we can construct a QBD-RAP process from \(Z\) with generator \(\bs B\) and closing operator \(\bs V\) with
	\[||\bs B \bs V f(x,i) - Bf(x,i)||_\infty < \delta.\]
	Further, \(\mathcal C\) is a core for \(B\).
\end{theorem}
\begin{proof}
	Since \(f\) is differentiable with bounded derivative it is also bounded and Lipschitz continuous, \(|f(x,i)|\leq G\) and \(|f(x,i)-f(y,i)|\leq L|x-y|\) for some \(G,\,L\) (which may depend on \(f\)). 

	Consider 
	\begin{align*}
		||\bs B \bs V f(x,i) - Bf(x,i)||_\infty &= ||\bs B \bs V f(x,i) - Bf(\bar x_{i}(x),i) + Bf(\bar x_{i}(x),i)- Bf(x,i)||_\infty.
	\end{align*}
	By the triangle inequality
	\begin{align}
		||\bs B \bs V f(x,i) - Bf(x,i)||_\infty &\leq ||\bs B \bs V f(x,i) - Bf(\bar x_{i}(x),i)||_\infty + ||Bf(\bar x_{i}(x),i)- Bf(x,i)||_\infty.\label{eqn: bnd}
	\end{align}
	
	By the Lipschitz continuity of \(f\) and its derivative and the triangle inequality, then the second term is
	\begin{align*}
		||Bf(\bar x_{i}(x),i)- Bf(x,i)||_\infty &= ||\sum_{j\in\mathcal S}T_{ij}(f(\bar x_{i}(x),j)-f(x,j)) + c_i(f_x(\bar x_{i}(x),i)-f_x(x,i))||_\infty
		\\&\leq \sum_{j\in\mathcal S}|T_{ij}|\sup_{x}|f(\bar x_{i}(x),j)-f(x,j)| + |c_i|\sup_{x}|f_x(\bar x_{i}(x),i)-f_x(x,i)|
		\\&\leq \sum_{j\in\mathcal S}|T_{ij}|L\sup_{x}|\bar x_{i}(x)-x| + |c_i|L'\sup_{x}|\bar x_{i}(x)-x|
		\\&= \hat L\sup_{x}|\bar x_{i}(x)-x|,
	\end{align*}
	where \(\hat{L} = \sum_{j\in\mathcal S}|T_{ij}|L + |c_i|L'\). 
	
	Recall, from Lemma~\ref{lem: points}, for an ME with sufficiently small variance and sufficiently large order \(p\) of its minimal representation, we can choose \(t_1=0\) and \(2\var(Z)^{1/3} < t_2 < t_3 < ... < t_p < \Delta-\var(Z)^{1/3}<\Delta=t_{p+1}\), such that \(t_1,...,t_p\), are relatively prime to \(2\pi\sigma_i\), \(i=1,...,q\), where \(\sigma_i\) is the magnitude of the imaginary part of the complex conjugate pairs of eigenvalues of \(\bs S\), and
	\begin{align}
		\max\limits_{n=1,...,p}|t_n-t_{n+1}|<\frac{1}{2}\delta/\hat{L},\label{eqn: cond 2a}
	\end{align}
	and 
	\begin{align}
		\cfrac{\bs \alpha e^{\bs S t_n}\bs s}{\bs \alpha e^{\bs S t_n}\bs e} \leq \cfrac{\var(Z)^{1/3}}{{\frac{1}{4}\delta/\hat{L}}\,(1-\var(Z)^{1/3})}\label{eqn: cond 3b}
	\end{align}
	are satisfied. So assume such a representation for \(Z\), in which case \(\max_{n=1,...,p} |t_n-t_{n+1}|<\frac{1}{2}\delta/\hat{L}\) and therefore \(\sup_{x}|\bar x_{i}(x)-x|<\frac{1}{2}\delta/\hat{L}\) and hence the second term \(||Bf(\bar x_{i}(x),i)- Bf(x,i)||_\infty<\delta/2\). 
	
	By Lemma~\ref{lem: generators are close} the first term on the right-hand side of \eqref{eqn: bnd} is bounded above by \(M(1+\bs e_n\tr{} \bs s)\left({\var(Z)^{1/3}}/\left(1-\var(Z)^{1/3}\right)\right)\). By \eqref{eqn: cond 3b}, \(\bs e_n\tr{} \bs s=\cfrac{\bs \alpha e^{\bs S t_n}\bs s}{\bs \alpha e^{\bs S t_n}\bs e} \leq \cfrac{\var(Z)^{1/3}}{{\frac{1}{4}\delta/\hat{L}}\,(1-\var(Z)^{1/3})}\). Hence,
	\[||\bs B \bs V f(x,i) - Bf(\bar x_{i}(x),i)||_\infty \leq \left(M+ \cfrac{\var(Z)^{1/3}}{{\frac{1}{4}\delta}\,(1-\var(Z)^{1/3})/\hat{L}}\right)\cfrac{\var(Z)^{1/3}}{1-\var(Z)^{1/3}}.\]
	So, choose \(Z\) with sufficiently small variance such that \(\left(M+ \cfrac{\var(Z)^{1/3}}{{\frac{1}{4}\delta}\,(1-\var(Z)^{1/3})/\hat{L}}\right)\cfrac{\var(Z)^{1/3}}{1-\var(Z)^{1/3}}<\delta/2\). 
	
	Thus, we have shown that if \(\var(Z)\) is sufficiently small (and the order of the minimal representation of \(Z\) is sufficiently large) then we can find a representation of the distribution of \(Z\) such that the right-hand side of (\ref{eqn: bnd}) is less than \(\delta\).

	To complete the proof we show that \(\mathcal C\) is a core for \(B\). The space of polynomials is a subset of \(\mathcal C\). Let \(g\in \mathrm{Dom}(B)\), then \(g\) is continuously differentiable with respect to the spatial variable, \(x\). A basic result in analysis is that if a function \(g\) has a continuous derivative on \([0,b]\) then there exists a polynomial \(h(x)\) such that \(|g(x)-h(x)|<\varepsilon\) and \(|g_x(x)-h_x(x)|<\varepsilon\) for all \(x\in[0,b]\) (see, for example, \cite[Exercise 6.7.11]{abbott}). Hence, there exist polynomials \(h_j(x)\) such that \(|g(x,j)-h_j(x)|<\varepsilon\) and \(|g_x(x,j)-(h_j)_x(x)|<\varepsilon\). Moreover, let \(h(x,j)=h_j(x)\), then 
	\begin{align*}
		||Bg(x,i)-Bh(x,i)||_{\infty} &= ||\sum_{j\in\mathcal S} T_{ij}(g(x,j)-h(x,j)) + c_j(g_x(x,j)-h_x(x,j))||_{\infty}
		\\&\leq \sum_{j\in\mathcal S} |T_{ij}|\sup_{x}|g(x,j)-h(x,j)| + |c_j|\sup_{x}|g_x(x,j)-h_x(x,j)| 
		\\&\leq \varepsilon \left(\sum_{j\in\mathcal S} |T_{ij}| + |c_j|\right).
	\end{align*}
	Hence, \(h(x,j)\) can be found such that \((h,Bh)\) is arbitrarily close to \((g,Bg)\) and so the space of polynomials is a core for \(B\) and hence, so is \(\mathcal C\). 
\end{proof}

To prove convergence of the scheme we utilise Theorem~6.1,~Chapter~1 of \cite{ethierkurtz}, which we now recall the relevant parts of, for completeness. First, let us set up notation in the current context. For \(p=1,2,...\), let \(\mathcal L^{(p)} = \mathbb R^{(K+1)\times|S|\times p}\), and \(\mathcal L\), be Banach spaces (with the supremum norm denoted by \(||\cdot||_{\infty}\)) and let \(\bs V^{(p)}:\mathcal L\to \mathcal L^{(p)}\) be a sequence of bounded linear transformations.
\begin{theorem}[Theorem~6.1,~Chapter~1 of \cite{ethierkurtz}]\label{thm: ek}
	For \(p=1,2,...,\) let \(\{\bs T^{(p)}(t)\}\) and \(\{T(t)\}\) be strongly continuous contraction semigroups on \(\mathcal L^{(p)}\) and \(\mathcal L\) with generators \(\bs B^{(p)}\) and \(B\). Let \(\mathcal C\) be a core of \(B\). Then the following are equivalent:
	\begin{enumerate}
		\item[(a)] For each \(f\in \mathcal L\), \(\bs T^{(p)}(t)\bs V^{(p)} f \to T(t)f\) for all \(t\geq 0\), uniformly on bounded intervals.
		\item[(b)] Omitted.
		\item[(c)] For each \(f\in\mathcal C\), there exists \(f^{(p)}\in \mathrm{Dom}(\bs B^{(p)})\) for each \(p\geq 1\) such that \(f^{(p)}\to f\) and \(\bs B^{(p)}f^{(p)}\to Bf\).
	\end{enumerate}
\end{theorem}

\begin{theorem}[Convergence of the QBD-RAP to the fluid queue]\label{thm: main convergence}
	Let \(\{Z^{(p)}\}_{p}\) be a sequence of matrix exponential random variables with mean \(\Delta\), \(\var(Z^{(p)})\to 0\) as \(p\to\infty\), and increasing order. Further, let \(\{(L^{(p)}(t), \bs A^{(p)}(t), \phi(t))\},\,p=1,2,...,\) be the corresponding sequence of QBD-RAP approximations constructed as in Section~\ref{sec:QBD-RAP Approximation} with generators \(\bs B^{(p)}\) and closing operators \(\bs V^{(p)}\). Then the QBD-RAP scheme converges weakly to the distribution of the fluid queue as \(p\to\infty\).
\end{theorem}
\begin{proof}
	Consider a sequence of matrix exponential variables \(\{Z^{(p)}\}_{p}\) with mean \(\Delta\), decreasing variance (decreasing coefficient of variation) and increasing minimal order. Such a sequence exists; consider the sequence of Erlang distributions with mean \(\Delta\) and \(p\) phases. 
	
	We show weak convergence via Part~(c) of Theorem~\ref{thm: ek} and hence we must show two things: there exists \(f^{(p)}\in\mathcal L^{(p)}\) with \(f^{(p)}\to f\), and \(\bs B^{(p)}f^{(p)} \to Bf\). For the former, construct \(f^{(p)}\) as \(f_p\) from Corollary~\ref{eqn: fns converge}. For the latter, Theorem~\ref{thm: generators converge} shows that, for any \(\delta>0\), we can find \(p_0\) sufficiently large such that for every \(p>p_0\) we can find a representation of \(Z^{(p)}\) (we are free to choose the representation of \(Z^{(p)}\)) and an associated QBD-RAP approximation scheme such that \(||\bs B^{(p)}f^{(p)}-Bf||_{\infty}<\delta\), and hence \(||\bs B^{(p)}f^{(p)}-Bf||_{\infty}\to 0\) as \(p\to \infty\).
	
	Now, apply Theorem~\ref{thm: ek} which states that Part~(a) also holds, so \(\bs T^{(p)}(t)f^{(p)}\to T(t)f\), which is equivalent to \(\mathbb E\left[\displaystyle\sum_{\ell=1}^{K+1}\sum_{i\in\mathcal S}\bs A^{(p)}(t)1(L^{(p)}(t)=\ell, \phi(t)=i)\bs V^{(p)}_{\ell,i}f\right]\to \mathbb E[f(X(t),\varphi(t))]\), and this completes the proof. 
\end{proof}

%% file: conclusion.tex
\section{Conclusion}
\label{sec:Conclusion}

In this paper we constructed a QBD-RAP approximation to a fluid queue, described its dynamics, and proved its convergence. Because the QBD-RAP is itself a stochastic process, every approximation to a probability that it produces is guaranteed to have the properties of a true probability: non-negative, bounded above by 1, and integrating to 1.

The construction discretises the state space of the fluid queue into intervals and, within each interval, approximates deterministic events using (concentrated) matrix exponential distributions. The resulting phase process has the same distribution as the phase process of the fluid queue, and the level process approximates the distribution of which interval contains the fluid level. To recover an approximation to the distribution of the fluid queue \emph{within} each interval, we introduced the closing operator, which uses the orbit process of the QBD-RAP together with bases of conditional residual time distributions.

Our main distributional result establishes convergence of the generator of the QBD-RAP to the generator of the fluid queue, with respect to any sequence of matrix exponential distributions whose coefficient of variation tends to zero at fixed mean. No particular sequence is required, only that the coefficient of variation vanishes and the minimal order grows without bound, and the proof method itself is novel, analysing the QBD-RAP via bases of conditional residual time distributions rather than the orbit process used in prior analyses of RAP-modulated processes. Together with the fact that the QBD-RAP retains a genuine probabilistic interpretation, this convergence result gives the first generator-theoretic guarantee that an approximation of this kind converges to the fluid queue while preserving the properties of true probabilities throughout. Numerical results demonstrating the accuracy of the approximation, and comparing it against existing methods, will be presented in subsequent work.

%% file: proof.tex
\section{Bounds for closing operators.}
For \(j\in\mathcal S,\,\ell\in\{1,...,K\}\), define \(g_{\ell,j}(x)=f(y_\ell+x, j)1(j\in\mathcal P) + f(y_{\ell+1}-x, j)1(j\in\mathcal N)\) and suppose \(g_{\ell,j}:[0,\Delta] \to \mathbb R\) is bounded and Lipschitz continuous, \(|g_{\ell,j}(x)|\leq G_{\ell,j}\leq G\),~\(|g_{\ell,j}(x)-g_{\ell,j}(y)|\leq L_{\ell,j}|x-y|\leq L|x-y|,\) \(\ell=1,...,K,\,j\in\mathcal S\). Consider
\begin{align*}
	&\left|\bs e_n\tr{}\bs V_{\ell,j}f - g_{\ell,j}(t_n)\right|
	\\&= \left|\bs e_n\tr{}\int_{[0,\infty)} e^{ \bs{S}x} \bs s \left(g_{\ell,j}(|x\,\mathrm{mod }\, 2\Delta - \Delta|) - g_{\ell,j}(t_n)\right)\wrt x\right|
	\\&= \left|\int_{[0,\infty)} \cfrac{\mathbb P(Z-t_n\in [x, x+\wrt x])}{\mathbb P(Z\geq t_n)} \left(g_{\ell,j}(|x\,\mathrm{mod }\, 2\Delta - \Delta|) - g_{\ell,j}(t_n)\right)\right|
	\\&\leq \int_{[0,\infty)} \cfrac{\mathbb P(Z-t_n\in [x, x+\wrt x])}{\mathbb P(Z\geq t_n)} \left|\left(g_{\ell,j}(|x\,\mathrm{mod }\, 2\Delta - \Delta|) - g_{\ell,j}(t_n)\right)\right|
	\\&= \int_{[0,\infty)\setminus (\Delta-t_n-\varepsilon,\Delta-t_n+\varepsilon)} \cfrac{\mathbb P(Z-t_n\in [x, x+\wrt x])}{\mathbb P(Z\geq t_n)} \left|\left(g_{\ell,j}(|x\,\mathrm{mod }\, 2\Delta - \Delta|) - g_{\ell,j}(t_n)\right)\right| 
	\\&\quad{} + \int_{[\Delta-t_n-\varepsilon,\Delta-t_n+\varepsilon]} \cfrac{\mathbb P(Z-t_n\in [x, x+\wrt x])}{\mathbb P(Z\geq t_n)} \left|\left(g_{\ell,j}(|x - \Delta|) - g_{\ell,j}(t_n)\right)\right|.
\end{align*}
Since we chose \(t_n<\Delta-\varepsilon\), then \(\mathbb P(Z\geq t_n)\geq 1-\cfrac{\var(Z)}{\varepsilon^2}\), by Chebyshev's inequality. Hence, the second term in the last line above is less than or equal to 
\begin{align*}
	&\int_{[\Delta-t_n-\varepsilon,\Delta-t_n+\varepsilon]} \cfrac{\mathbb P(Z-t_n\in [x, x+\wrt x])}{1-{\var(Z)}/{\varepsilon^2}} \left|\left(g_{\ell,j}(|x - \Delta|) - g_{\ell,j}(t_n)\right)\right|
	\\&\leq \int_{[\Delta-t_n-\varepsilon,\Delta-t_n+\varepsilon]} \cfrac{\mathbb P(Z-t_n\in [x, x+\wrt x])}{1-{\var(Z)}/{\varepsilon^2}} L\varepsilon 
	\\&\leq \cfrac{L\varepsilon}{1-{\var(Z)}/{\varepsilon^2}}
\end{align*}
where the first inequality is due to the Lipschitz continuity of \(g_{\ell,j}\). The first term is less than or equal to 
\begin{align*}
	&\int_{[0,\infty)\setminus (\Delta-t_n-\varepsilon,\Delta-t_n+\varepsilon)} \cfrac{\mathbb P(Z-t_n\in [x, x+\wrt x])}{1-{\var(Z)}/{\varepsilon^2}} \left|\left(g_{\ell,j}(|x\,\mathrm{mod }\, 2\Delta - \Delta|) - g_{\ell,j}(t_n)\right)\right| 
	\\&\leq \int_{[0,\infty)\setminus (\Delta-t_n-\varepsilon,\Delta-t_n+\varepsilon)} \mathbb P(Z-t_n\in [x, x+\wrt x])\cfrac{2G}{1-{\var(Z)}/{\varepsilon^2}} 
	\\&=\mathbb P(Z\notin (\Delta-\varepsilon,\Delta+\varepsilon))\cfrac{2G}{1-{\var(Z)}/{\varepsilon^2}}
	\\&\leq \cfrac{\var(Z)}{\varepsilon^2}\cfrac{2G}{1-{\var(Z)}/{\varepsilon^2}},
\end{align*}
where the first inequality is from the boundedness of \(g_{\ell,j}\) and the last inequality is from Chebyshev's inequality. Hence, 
\[
	\left|\bs e_n\tr{}\bs V_{\ell,j}g - g_{\ell,j}(t_n)\right| \leq \cfrac{\var(Z)}{\varepsilon^2}\cfrac{2G}{1-{\var(Z)}/{\varepsilon^2}} + \cfrac{L\varepsilon}{1-{\var(Z)}/{\varepsilon^2}}
\]
Setting \(\varepsilon=\var(Z)^{1/3}\) (as in \cite{hhat2020}), then when \(\var(Z)\) is small, 
\[
	\cfrac{\var(Z)}{\varepsilon^2}\cfrac{2G}{1-{\var(Z)}/{\varepsilon^2}}+\cfrac{L\varepsilon}{1-{\var(Z)}/{\varepsilon^2}}=\cfrac{\var(Z)^{1/3}(2G+L)}{1-\var(Z)^{1/3}}
\]
is also small. Hence, we have shown that \(\bs e_n\tr{}\bs V_{\ell,j}f,\) approximates the value \(f\) at \((y_\ell + t_n,j),\, j\in\mathcal P\) or \((y_{\ell+1} - t_n,j),\, j\in\mathcal N\), as per the following lemma.
\begin{lemma}\label{lem: cv evaluates g}
	For \(j\in\mathcal S,\,\ell\in\{1,...,K\},\) let \(g_{\ell,j}(x)=f(y_\ell+x, j)1(j\in\mathcal P) + f(y_{\ell+1}-x, j)1(j\in\mathcal N)\), \(x\in [0,\Delta]\), and suppose \(g_{\ell,j}:[0,\Delta] \to \mathbb R\) is bounded and Lipschitz continuous, \(|g_{\ell,j}(x)|\leq G\),~\(|g_{\ell,j}(x)-g_{\ell,j}(y)|\leq L|x-y|\). Then 
	\[
		\bs e_n\tr{}\bs V_{\ell,j}f = g_{\ell,j}(t_n)+r^1_{\ell,j,n} = \begin{cases}
			f(y_\ell+t_n, j) + r^1_{\ell,j,n}, & j\in\mathcal P,
			\\ f(y_{\ell+1}-t_n, j) + r^1_{\ell,j,n}, & j\in\mathcal N,
		\end{cases}
	\]
	for \(n=1,...,p\), where \(|r^1_{\ell,j,n}|\leq \cfrac{\var(Z)^{1/3}(2G+L)}{1-\var(Z)^{1/3}}\).
\end{lemma}
A direct corollary is a bound on the accuracy of the approximation of \(f\) by the simple function \(f_p\).
\begin{corollary}\label{cor: Vf approximates f}
	Let \(f:[0,b]\times \mathcal S \to \mathbb R\) be bounded and Lipschitz continuous, \(|f(x,j)|\leq G\) and \(|f(x,j)-f(y,j)|\leq L|x-y|\), and define \(f_p(x,j)=\bs Vf(x,j)\), then 
	\begin{align*}
		f_p(x, j) &
		= \begin{cases}
			f(0,j), & j\in\mathcal N, x=0,
			\\f(y_{\ell}+t_n,j) + r_{\ell,j,n}^1, & j\in\mathcal P,\,x\in \left[y_{\ell}+t_n,y_{\ell}+t_{n+1}\right),\, n=1,...,p,
			\\ f(y_{\ell+1}-t_n,j)+ r_{\ell,j,n}^1,& j\in\mathcal N,\,x\in \left(y_{\ell+1}-t_{n+1},y_{\ell+1}-t_{n}\right],\, n=1,...,p,
			\\f(b,j), & j\in\mathcal P, x=b,
		\end{cases}
	\end{align*}
	where \(|r^1_{\ell,j,n}|\leq\cfrac{\var(Z)^{1/3}(2G+L)}{1-\var(Z)^{1/3}}\).
\end{corollary}

In addition to the above, the proof that the generator converges requires the following results. Suppose further that \(g_{\ell,j}:[0,\Delta] \to \mathbb R\) is differentiable on \((0,\Delta)\). Consider
\begin{align*}
	\bs e_n\tr{}\bs S\bs V_{\ell,j}f &= \bs e_n\tr{}\bs S\int_{[0,\infty)} e^{ \bs{S}x} \bs s g_{\ell,j}(|x\,\mathrm{mod }\, 2\Delta - \Delta|)\wrt x
	\\&= \sum_{k=0}^\infty \bs e_n\tr{}\bs S\int_{[0,\Delta)} \left( e^{ \bs{S}(k2\Delta+x)} \bs s + e^{ \bs{S}((k+1)2\Delta-x)} \bs s \right)g_{\ell,j}(\Delta-x)\wrt x.
\end{align*}
Integrating by parts, then 
\begin{align*}
	\bs e_n\tr{}\bs S\bs V_{\ell,j}f &= \sum_{k=0}^\infty \bs e_n\tr{} \left.\left( e^{ \bs{S}(k2\Delta+x)} \bs s - e^{ \bs{S}((k+1)2\Delta-x)} \bs s \right)g_{\ell,j}(\Delta-x)\right|_{0}^{\Delta}
	\\&\quad{} + \sum_{k=0}^\infty \bs e_n\tr{}\int_{[0,\Delta)} \left( e^{ \bs{S}(k2\Delta+x)} \bs s - e^{ \bs{S}((k+1)2\Delta-x)} \bs s \right)(g_{\ell,j})_x(\Delta-x)\wrt x.
\end{align*}
The first series above is a telescoping sum and equals 
\[
	-\bs e_n\tr{} \bs s g_{\ell,j}(\Delta)=\begin{cases}
		-\bs e_n\tr{} \bs s f(y_{\ell+1},j), & j\in\mathcal P,
		\\-\bs e_n\tr{} \bs s f(y_\ell,j), & j\in\mathcal N,
	\end{cases}
\]
the second series is equal to 
\[
	\bs e_n\tr{}\bs V_{\ell,j}(g_{\ell,j})_x = \begin{cases}
		\bs e_n\tr{}\bs V_{\ell,j}f_x, & j\in\mathcal P,
		\\-\bs e_n\tr{}\bs V_{\ell,j}f_x, & j\in\mathcal N,
	\end{cases}
\]
to which we can apply Lemma~\ref{lem: cv evaluates g}. Hence, we have the following corollary.
\begin{corollary}\label{cor: S cv evaluates g'}
	For \(j\in\mathcal S\), \(\ell\in\{1,...,K\}\), let \(g_{\ell,j}(x)=f(y_\ell+x, j)1(j\in\mathcal P) + f(y_{\ell+1}-x, j)1(j\in\mathcal N)\) and suppose \(g_{\ell,j}:[0,\Delta] \to \mathbb R\) is differentiable on \((0,\Delta)\) and the derivative is bounded and Lipschitz continuous, \(|(g_{\ell,j})_x(x)|\leq G'\),~\(|(g_{\ell,j})_x(x)-g_{\ell,j}'(y)|\leq L'|x-y|\). Then 
	\[
		\bs e_n\tr{}\bs S\bs V_{\ell,j}f = \begin{cases}
			-\bs e_n\tr{} \bs s f(y_{\ell+1},j)+ f_x(y_\ell + t_n,j)+ r{}'{}^{1}_{\ell,j,n}, & j\in\mathcal P,
			\\-\bs e_n\tr{} \bs s f(y_\ell,j)- f_x(y_{\ell+1}-t_n,j)+ r{}'{}^{1}_{\ell,j,n}, & j\in\mathcal N,
		\end{cases}
	\]
	for \(n=1,...,p\), where \(|r{}'{}^{1}_{\ell,j,n}|\leq\cfrac{\var(Z)^{1/3}(2G'+L')}{1-\var(Z)^{1/3}}\).
\end{corollary}

The following result is also used to prove convergence.
\begin{lemma}\label{lem: D cv evaluates g}
	For \(j\in\mathcal S,\,\ell\in\{1,...,K\},\) let \(g_{\ell,j}(x)=f(y_\ell+x, j)1(j\in\mathcal P) + f(y_{\ell+1}-x, j)1(j\in\mathcal N)\), \(x\in [0,\Delta]\), and suppose \(g_{\ell,j}:[0,\Delta] \to \mathbb R\) is bounded and Lipschitz continuous, \(|g_{\ell,j}(x)|\leq G\),~\(|g_{\ell,j}(x)-g_{\ell,j}(y)|\leq L|x-y|\). Then 
	\begin{align}
		\bs e_n\tr{}\bs D \bs V_{\ell,j}f = g_{\ell,j}(\Delta-t_n)+r^4_{\ell,j,n} 
		= \begin{cases}
			f(y_{\ell+1}-t_n, j) + r^4_{\ell,j,n}, & j\in\mathcal P,
			\\ f(y_{\ell}+t_n, j) + r^4_{\ell,j,n}, & j\in\mathcal N,
		\end{cases}
	\end{align}
	for \(n=1,...,p\), where \(|r^4_{\ell,j,n}|\leq M \cfrac{\var(Z)^{1/3}}{1-\var(Z)^{1/3}}\), for some \(M<\infty\) independent of \(\var(Z),\, \ell,\, j,\, n,\) (but dependent on \(G\) and \(L\)). 
\end{lemma}
\begin{proof}
    Using the definition of \(\bs D\),
    \begin{align}
        &\left|\bs e_n\tr{} \bs D\bs V_{\ell,j}f - g_{\ell,j}(\Delta-t_n)\right|  \nonumber
        \\&= \left|\bs e_n\tr{} \left[\int_{x=0}^{\Delta-\varepsilon}e^{\bs Sx}\bs s\cfrac{\bs \alpha e^{\bs Sx}}{\bs\alpha e^{\bs S x}\bs e}\wrt x + \int_{x=\Delta-\varepsilon}^\infty e^{\bs Sx}\bs s \cfrac{\bs \alpha e^{\bs S(\Delta-\varepsilon)}}{\bs\alpha e^{\bs S (\Delta-\varepsilon)}\bs e} \wrt x\right]\bs V_{\ell,j}f - g_{\ell,j}(\Delta-t_n)\right|  \nonumber
        \\&\leq \left|\bs e_n\tr{} \int_{x=0}^{\Delta-\varepsilon}e^{\bs Sx}\bs s\cfrac{\bs \alpha e^{\bs Sx}}{\bs\alpha e^{\bs S x}\bs e}\wrt x \bs V_{\ell,j}f - g_{\ell,j}(\Delta-t_n)1(n>1)\right| \nonumber
        \\& \quad{} 
        + \left|\bs e_n\tr{}\int_{x=\Delta-\varepsilon}^\infty e^{\bs Sx}\bs s\cfrac{\bs \alpha e^{\bs S(\Delta-\varepsilon)}}{\bs\alpha e^{\bs S (\Delta-\varepsilon)}\bs e} \wrt x\bs V_{\ell,j}f - g_{\ell,j}(\Delta-t_n)1(n=1)\right|,\label{eqn: bnd this 2}
    \end{align}
    by the triangle inequality. The second term in (\ref{eqn: bnd this 2}) is equal to 
    \begin{align}
        \left|\cfrac{\mathbb P(Z>t_n+\Delta-\varepsilon)}{\mathbb P (Z>t_n)}\cfrac{\bs \alpha e^{\bs S(\Delta-\varepsilon)}}{\bs\alpha e^{\bs S (\Delta-\varepsilon)}\bs e} \bs V_{\ell,j}f - g_{\ell,j}(\Delta-t_n)1(n=1)\right|.\label{eqn: bnd this 1}
    \end{align}
    Since we chose \(t_n<\Delta-\varepsilon\), then by Chebyshev's inequality, 
    \(\mathbb P(Z>t_n) \geq 1-\var(Z)/\varepsilon^2\). Also, by Chebyshev's inequality, since \(t_n>2\varepsilon\) for \(n=2,...,p\), then \(\mathbb P(Z>t_n+\Delta-\varepsilon)\leq \mathbb P(Z>\Delta+\varepsilon)\leq \var(Z)/\varepsilon^2\). Hence, for \(n\geq 2\),  (\ref{eqn: bnd this 1}) is less than or equal to 
    \begin{align}
        \left|\cfrac{\var(Z)/\varepsilon^2}{1-\var(Z)/\varepsilon^2}\cfrac{\bs \alpha e^{\bs S(\Delta-\varepsilon)}}{\bs\alpha e^{\bs S (\Delta-\varepsilon)}\bs e} \bs V_{\ell,j}f\right|\leq \cfrac{\var(Z)/\varepsilon^2}{1-\var(Z)/\varepsilon^2}G,
    \end{align}
    since \(\bs V_{\ell,j}\) is a bounded operator, \(f\) a bounded function.

    Now, replacing \(t_n\) by \(x\leq \Delta-\varepsilon\) in the proof of Lemma~\ref{lem: cv evaluates g} shows that 
    \begin{align}
        \displaystyle\cfrac{\bs \alpha e^{\bs Sx}}{\bs\alpha e^{\bs S x}\bs e}\bs V_{\ell,j}f = g_{\ell,j}(x) + r_{\ell,j,x}^2, \label{eqn: cv evaluates g}
    \end{align}
    where 
    \(r_{\ell,j,x}^2 \leq \cfrac{\var(Z)^{1/3}(2G+L)}{1-\var(Z)^{1/3}}\). Hence, \(\cfrac{\bs \alpha e^{\bs S(\Delta-\varepsilon)}}{\bs\alpha e^{\bs S (\Delta-\varepsilon)}\bs e} \bs V_{\ell,j}f - g_{\ell,j}(\Delta-\varepsilon)=r^1_{\ell,j,\Delta-\varepsilon}\). Therefore, for \(n=1\), the second term in (\ref{eqn: bnd this 2}) is
    \begin{align*}
        &\left|\bs e_1\tr{}\int_{x=\Delta-\varepsilon}^\infty e^{\bs Sx}\bs s\cfrac{\bs \alpha e^{\bs S(\Delta-\varepsilon)}}{\bs\alpha e^{\bs S (\Delta-\varepsilon)}\bs e} \wrt x\bs V_{\ell,j}f - g_{\ell,j}(\Delta-\varepsilon) + g_{\ell,j}(\Delta-\varepsilon) - g_{\ell,j}(\Delta)\right| 
        \\&= \left|\bs e_1\tr{}\int_{x=\Delta-\varepsilon}^\infty e^{\bs Sx}\bs s\wrt xr^1_{\ell,j,\Delta-\varepsilon} + g_{\ell,j}(\Delta-\varepsilon) - g_{\ell,j}(\Delta)\right| 
        \\& \leq |r_{\ell,j,\Delta-\varepsilon}^1| + L\varepsilon,
    \end{align*}
    since \(\bs e_1\tr{}\int_{x=\Delta-\varepsilon}^\infty e^{\bs Sx}\bs s\leq 1\) and by the Lipschitz continuity of \(g\) and the triangle inequality.
    
    By (\ref{eqn: cv evaluates g}), the first term in (\ref{eqn: bnd this 2}) is
    \begin{align}
        &\left|\bs e_n\tr{} \int_{x=0}^{\Delta-\varepsilon}e^{\bs Sx}\bs s\left(g_{\ell,j}(x) + r_{\ell,j,x}^2 \right)\wrt x - g_{\ell,j}(\Delta-t_n)1(n>1)\right|\nonumber
        \\&\leq \left|\bs e_n\tr{} \int_{x=0}^{\Delta-\varepsilon}e^{\bs Sx}\bs sg_{\ell,j}(x)\wrt x - g_{\ell,j}(\Delta-t_n)1(n>1)\right| + \left|\bs e_n\tr{} \int_{x=0}^{\Delta-\varepsilon}e^{\bs Sx}\bs sr_{\ell,j,x}^2 \wrt x\right|.\label{eqn: bnd this 3}
    \end{align}
    The second term in (\ref{eqn: bnd this 3}) less than or equal to \(\cfrac{\var(Z)^{1/3}(2G+L)}{1-\var(Z)^{1/3}}\) since \(|r_{\ell,j,x}^2|\) is bounded independently of \(x\) by \(\cfrac{\var(Z)^{1/3}(2G+L)}{1-\var(Z)^{1/3}}\) and \(\bs e_n\tr{} \int_{x=0}^{\Delta-\varepsilon}e^{\bs Sx}\bs s=\mathbb P(Z-t_n<\Delta-\varepsilon\mid Z>t_n)\leq 1\). For \(n=1\), the first term in (\ref{eqn: bnd this 3}) is 
    \[
        \left|\bs e_1\tr{} \int_{x=0}^{\Delta-\varepsilon}e^{\bs Sx}\bs sg_{\ell,j}(x)\wrt x \right|\leq G\mathbb P(Z<\Delta-\varepsilon)\leq G\cfrac{\var(Z)}{\varepsilon^2},
    \]
    by Chebyshev's inequality. For \(n>1\), the first term in (\ref{eqn: bnd this 3}) is 
    \begin{align}
        &\left|\bs e_n\tr{} \int_{x=0}^{\Delta-\varepsilon}e^{\bs Sx}\bs sg_{\ell,j}(x)\wrt x - g_{\ell,j}(\Delta-t_n)\right| \nonumber
        \\&= \left|\bs e_n\tr{} \int_{\substack{x\in [0,\Delta-\varepsilon]\\x\notin (\Delta-t_n-\varepsilon,\Delta-t_n+\varepsilon)}}e^{\bs Sx}\bs sg_{\ell,j}(x)\wrt x \right| 
        + \left|\bs e_n\tr{} \int_{x=\Delta-t_n-\varepsilon}^{\Delta-t_n+\varepsilon}e^{\bs Sx}\bs sg_{\ell,j}(x)\wrt x - g_{\ell,j}(\Delta-t_n)\right|. \label{eqn: bnd this 4}
    \end{align}
    Since \(g_{\ell,j}\) is bounded, then the first term in (\ref{eqn: bnd this 4}) is less than or equal to 
    \[
        G\mathbb P(Z\notin(\Delta-\varepsilon,\Delta+\varepsilon)\mid Z>t_n)\leq G\cfrac{\var(Z)/\varepsilon^2}{1-\var(Z)/\varepsilon^2}.
    \]
    Since \(g_{\ell,j}\) is Lipschitz continuous then 
    \begin{align*}
        \bs e_n\tr{} \int_{x=\Delta-t_n-\varepsilon}^{\Delta-t_n+\varepsilon}e^{\bs Sx}\bs sg_{\ell,j}(x)\wrt x = \bs e_n\tr{} \int_{x=\Delta-t_n-\varepsilon}^{\Delta-t_n+\varepsilon}e^{\bs Sx}\bs s\left(g_{\ell,j}(\Delta-t_n)+r^3_{\ell,j,n}(x)\right)\wrt x,
    \end{align*}
    where \(r^3_{\ell,j,n}(x)\) is also Lipschitz continuous with the same constant \(L\). Now, since \(r^3_{\ell,j,n}(x)\) is Lipschitz continuous, then \(|r^3_{\ell,j,n}(x)-r^3_{\ell,j,n}(\Delta-\varepsilon)|\leq L\varepsilon\) for all \(x\in[\Delta-\varepsilon,\Delta+\varepsilon]\), so 
    \[-\bs e_n\tr{} \int_{x=\Delta-t_n-\varepsilon}^{\Delta-t_n+\varepsilon}e^{\bs Sx}\bs sL\varepsilon\wrt x\leq \bs e_n\tr{}\int_{x=\Delta-t_n-\varepsilon}^{\Delta-t_n+\varepsilon} e^{\bs Sx}\bs sr^3_{\ell,j,n}(x)\wrt x\leq\bs e_n\tr{} \int_{x=\Delta-t_n-\varepsilon}^{\Delta-t_n+\varepsilon}e^{\bs Sx}\bs sL\varepsilon\wrt x.\] 
    Also, 
    \begin{align*}
        &\left|\bs e_n\tr{} \int_{x=\Delta-t_n-\varepsilon}^{\Delta-t_n+\varepsilon}e^{\bs Sx}\bs s\wrt xg_{\ell,j}(\Delta-t_n) - g(\Delta-t_n)\right|
        \\&=\left|\mathbb P(Z\in(\Delta-\varepsilon,\Delta+\varepsilon)\mid Z>t_n) g_{\ell,j}(\Delta-t_n) - g(\Delta-t_n)\right|
        \\&=\left|\mathbb P(Z\notin(\Delta-\varepsilon,\Delta+\varepsilon)\mid Z>t_n) g_{\ell,j}(\Delta-t_n)\right|
        \\&\leq\cfrac{\var(Z)/\varepsilon^2}{1-\var(Z)/\varepsilon^2}G,
    \end{align*}
    by Chebyshev's inequality and the boundedness of \(g_{\ell,j}\). So the second term in (\ref{eqn: bnd this 4}) is 
    \begin{align*}
        &\left|\bs e_n\tr{} \int_{x=\Delta-t_n-\varepsilon}^{\Delta-t_n+\varepsilon}e^{\bs Sx}\bs sg_{\ell,j}(x)\wrt x - g_{\ell,j}(\Delta-t_n)\right|
        \\&= \left|\bs e_n\tr{} \int_{x=\Delta-t_n-\varepsilon}^{\Delta-t_n+\varepsilon}e^{\bs Sx}\bs s\left(g_{\ell,j}(\Delta-t_n)+r^3_{\ell,j,n}(x)\right)\wrt x - g_{\ell,j}(\Delta-t_n)\right|
        \\&\leq \left|\bs e_n\tr{} \int_{x=\Delta-t_n-\varepsilon}^{\Delta-t_n+\varepsilon}e^{\bs Sx}\bs sg_{\ell,j}(\Delta-t_n)\wrt x - g_{\ell,j}(\Delta-t_n)\right|
        + \left|\bs e_n\tr{} \int_{x=\Delta-t_n-\varepsilon}^{\Delta-t_n+\varepsilon}e^{\bs Sx}\bs sr^3_{\ell,j,n}(x)\wrt x\right|
        \\&\leq \cfrac{\var(Z)/\varepsilon^2}{1-\var(Z)/\varepsilon^2}G + L\varepsilon.
    \end{align*}
    Now choose \(\varepsilon=\var(Z)^{1/3}\) and assemble the established bounds to get the result.
    \end{proof}